\documentclass[12pt,reqno]{amsart}
\usepackage{amssymb}
\usepackage{amsmath}
\usepackage{hyperref}
\usepackage{color}
\usepackage{graphics}
\usepackage{graphicx}
\usepackage{tikz}
\usepackage{amsfonts}
\usepackage{amscd}
\usepackage{mathrsfs}
\usepackage{lipsum}
\usepackage{enumerate}
\usepackage{float} 
\allowdisplaybreaks
\usepackage{verbatim}
\usepackage{enumitem}
\usepackage{amsthm}
\allowdisplaybreaks
\usepackage[nodisplayskipstretch]{setspace}
\newcommand{\what}{\widehat}%
\newcommand{\wtilde}{\widetilde}%
\newcommand{\R}{\mathbb R}%
\newcommand{\C}{\mathbb C}%
\newcommand{\liea}{\mathfrak a}
\newcommand{\N}{\mathbb N}%
\newcommand{\hc}{\mathsf{c}}
\newcommand{\hb}{\mathsf{b}}

\newcommand{\m}{\mathfrak  }

\newcommand{\no}{\nonumber}

\newcommand{\la}{\lambda}

\newenvironment{rcases}
  {\left.\begin{aligned}}
  {\end{aligned}\right\rbrace}

\newtheorem{thm}{Theorem}[section]

\newtheorem{rem}[thm]{Remark}

\newtheorem{proposition}[thm]{Proposition}
\theoremstyle{definition}

\theoremstyle{remark}
\theoremstyle{remark}

\numberwithin{equation}{section}

\begin{document}

\author[M. Naik]{Muna Naik}
\address[Muna Naik]{
Indian Institute of Science, CV Raman Rd, Bengaluru, Karnataka, 560012
India }
\email{mnaik41@gmail.com, muna@iisc.ac.in}

\author[S. K. Ray]{Swagato K. Ray}
\address[Swagato K. Ray]{Stat-Math Unit, Indian Statistical
Institute, 203 B. T. Rd., Calcutta 700108, India}\email{swagato@isical.ac.in}

\author[J. Sarkar]{Jayanta Sarkar}
\address[Jayanta Sarkar]{Department of Mathematics and Statistics,  Indian Institute of Science Education and Research Kolkata, Mohanpur-741246, Nadia, West Bengal, India}
\email{jayantasarkarmath@gmail.com, jayantasarkar@iiserkol.ac.in}

\title[Asymptotic behaviour of  solutions ]{$L^p$-asymptotic behaviour of  solutions of the heat equation on Riemannian symmetric spaces of noncompact type}


\maketitle

\begin{abstract}
For Riemannian symmetric spaces $X=G/K$ of noncompact type, we show that for all left $K$-invariant $f\in L^1(X)$, the functions $\|h_t\|_{L^p(X)}^{-1}(f\ast h_t-M_p(f)h_t)$ (with $h_t$ being the heat kernel of $X$) converges to zero in $L^p(X)$, $p\in [1,\infty]$, as $t\to\infty$, with the constant $M_p(f)$ depending only on $p$ and $f$. We also prove an analogous result for the fractional heat kernels $h_t^{\alpha}$, $\alpha\in (0,1)$. The above results have recently been proved for the important special cases $p=1$ and $\alpha= 1,\frac 12$.
\end{abstract}
\baselineskip18pt
\section{Introduction and main result}

This article mainly deals with generalizations of two results, proved initially for $L^p(\R^n)$.
The first such result was proved in \cite[Theorem 2.1]{Vaz-3} (see also \cite{Vaz-1}): if $f\in L^1(\R^n)$ then for all $p\in [1,\infty]$,
\begin{equation}\label{vrn}
\lim_{t\to\infty}\frac{\|f\ast h_t-M(f)h_t\|_{L^p(\R^n)}}{\|h_t\|_{L^p(\R^n)}}=0,
\end{equation}
where
\begin{equation*}
M(f)=\int_{\R^n}f(x)\,dx,\:\:\:\:\:\:h_t(x)=(4\pi t)^{-\frac{n}{2}}e^{-\frac{\|x\|^2}{4t}},\:\:x\in\R^n,
\end{equation*}
is the standard heat kernel of $\R^n$. Analogous results for the fractional heat equation on $\R^n$ was also proved by Vázquez \cite{Vaz-1}. We note that using the Fourier transform on $\R^n$ we can write $M(f)=\widehat{f}(0)$.

We refer the reader to Subsection \eqref{pe1} for all unexplained notations. We consider a  Riemannian symmetric spaces $X=G/K$ of noncompact type, where $G$ is a connected, noncompact, semisimple Lie group with finite center, and $K$ is a maximal compact subgroup. We will show that the exact analogue of \eqref{vrn} does not hold in this case for $p>1$.  Precisely, in order to establish an exact analogue of \eqref{vrn}, it is necessary to assume that \(f\) is \(K\)-biinvariant and to replace $M(f)$ by $M_p(f)$, the spherical Fourier transform $\what f(i \gamma_p \rho)$ of $f$ at $i \gamma_p \rho$, where $\displaystyle{\gamma_p=\frac{2}{p}-1}$, which depends on $p$. We have been dictated by the estimates for $L^p$ norms of the heat kernel $h_t$ of $X$ (see \cite[Proposition 4.1.1]{AL}): for all $t\in (1,\infty)$,
\begin{equation}
\label{ehe-1}
   \|h_t\|_p \asymp \left\{\begin{array}{lr}
        t^{-\frac{\ell}{2p'}}e^{-\frac{4}{pp'}|\rho|^2t}, & 1 \le p<2,\\
        t^{-\frac{\nu}{4}}e^{-{|\rho|^2t}}, &  p=2,\\
        t^{-\frac{\nu}{2}} e^{-{|\rho|^2t}}, &   2 < p \le \infty,
        \end{array}\right.
\end{equation}
where $\nu= \ell + 2|\Sigma_r^+|$, $\ell$ is the rank of $X$, and $|\Sigma_r^+|$ is the number of positive indivisible roots. These estimates can be obtained from the following global estimate of the heat kernel $h_t$ \cite{AL, AO}: for all $t>0$ and $H \in \overline{\m a_+}$
\begin{eqnarray}
\label{hes-1}
h_t(\exp H)
&\asymp& t^{-\frac n2}\left\lbrace{\prod_{\alpha \in \Sigma_r^+}}\left((1+\langle \alpha, H \rangle)(1+t+\langle \alpha, H \rangle)^{\frac{m_\alpha+m_{2\alpha}}{2}-1} \right) \right\rbrace\no\\&&\:\:\:\times e^{-|\rho|^2t-\langle \rho, H\rangle-\frac{|H|^2}{4t}}.
\end{eqnarray}
The main point to note here is that $\|h_t\|_p$, unlike in Euclidean spaces, is essentially independent of $p$ if $p\in(2,\infty]$.

Now we are in a position to state our main results. The first main result of this paper is the following.
\begin{thm}
\label{cl16}
\begin{enumerate}
\item [(a)] If $p\in [1,2]$ and $f \in \mathcal L^p(G//K)$, then
\begin{equation}
\lim_{t \to \infty}\|h_t\|_{L^p(X)} ^{-1}\|f \ast h_t -\what f(i \gamma_p \rho)h_t\|_{L^p(X)}=0.
\end{equation}
\item[(b)] If $p\in (2,\infty]$ and $f \in \mathcal L^2(G//K)$, then
\begin{equation}
\lim_{t \to \infty}\|h_t\|_{L^p(X)} ^{-1}\|f \ast h_t -\what f(0)h_t\|_{L^p(X)}=0.
\end{equation}
\end{enumerate}
\end{thm}
\begin{rem}
  \textup{We refer the reader to \eqref{funspace} and the subsequent discussion for the definition of $\mathcal{L}^p(G//K)$, as well as for the inclusion $\mathcal{L}^p(G//K) \ast L^p(G//K) \subset L^p(G//K)$ for all $p \in [1,2]$. We also refer to \eqref{point} and the subsequent discussion in the proof of Theorem~\ref{cl16} (b) for the explanation that for all large $t$, $f \ast h_t \in L^p(G//K)$ for all $p \in (2,\infty]$, whenever $f \in \mathcal{L}^2(G//K)$.
}
 \end{rem}
Using Theorem \ref{cl16} we then proceed to prove the following analogue of Theorem \ref{cl16} for the fractional heat kernels $h_t^{\alpha}$, $\alpha\in (0,1)$.
\begin{thm}
\label{fcl16}
\begin{enumerate}
\item [(a)] If $p\in [1,2]$ and $f \in \mathcal L^p(G//K)$, then
\begin{equation}
\lim_{t \to \infty}\|h_t^\alpha\|_{L^p(X)} ^{-1}\|f \ast h_t^\alpha -\what f(i \gamma_p \rho)h_t^\alpha\|_{L^p(X)}=0.
\end{equation}
\item[(b)] If $p\in (2,\infty]$ and $f \in \mathcal L^2(G//K)$, then
\begin{equation}
\lim_{t \to \infty}\|h_t^\alpha\|_{L^p(X)} ^{-1}\|f \ast h_t^\alpha -\what f(0)h_t^\alpha\|_{L^p(X)}=0.
\end{equation}
\end{enumerate}
\end{thm}
\begin{rem}
\textup{The fundamental difference between (\ref{vrn}) and Theorem \ref{cl16} is that the constant $M(f)=\what f(0)$, has been replaced by $\what f(i\gamma_p\rho)$, $\what f(0)$, for $p\in [1,2]$ and $p\in (2,\infty)$ respectively. It is not hard to see that $\what f(i \gamma_p \rho)$, $\what f(0)$ are the unique complex numbers for which Theorem \ref{cl16} holds. This is due to the following observation: for two families of functions $\{f_t:t>0\}$ and $\{g_t:t>0\}$ in $L^p(X)$ if we have
 \begin{equation*}
     \lim_{t\to\infty}\|g_t\|_p^{-1}\|f_t-\zeta g_t\|_p=0,
 \end{equation*}
for some $\zeta\in\C$, then $\zeta$ is unique. Because if the above limit is zero for some other $z\in\C$ then using the triangle inequality, we get
\begin{equation*}
    \|g_t\|_p^{-1}\|f_t-z g_t\|_p\geq \|g_t\|_p^{-1}\left(|z-\zeta|\|g_t\|_p-\|f_t-\zeta g_t\|_p\right),
\end{equation*}
which implies that $|z-\zeta|=0$, by taking limit as $t\to\infty$. The situation is similar for Theorem \ref{fcl16}. }
\end{rem}
Using the above results we will show in Remark \ref{irem-1} that for each $\alpha\in (0,1]$, the family $\{\|h_t^\alpha\|_p^{-1}h_t^\alpha \mid t>0\}$, where $h_t^1=h_t$, is an extremizer for the norm of the convolution operator $T_f$ for non-negative $f\in \mathcal L^p(G//K)$, $p\in [1,2]$.

Although the case $p=1$ of Theorem \ref{cl16} has been proved in \cite[Theorem 1.2]{AE}, we include it here for the sake of completeness and for its use in the proof of Theorem \ref{fcl16}. The proof of Theorem \ref{cl16} for the case $p \in (1,2)$ uses the method of the first proof of \cite[Theorem 1.2]{AE} along with the $L^p$-concentration of $h_t$ on $X$ proved in \cite[Theorem 4.1.2]{AL}. At this point, one may wonder why the constant $\what{f}(0)$ works for all $p\in [2,\infty]$ in Theorem \ref{cl16}. This can be attributed to the estimates \eqref{ehe-1} which implies that for all large $t$,  $\|h_t\|_{L^p(X)}$ is comparable to $\|h_{t/2}\|_{L^2(X)}^2$ and the Kunze--Stein phenomenon (Theorem \ref{ksphn}).

In recent times, several generalizations and variants of (\ref{vrn}) have been proved in \cite{Vaz-2, AE, Eff1, Eff2}. We will mostly concentrate on the papers by Vázquez \cite{Vaz-2}, which proves analogue of (\ref{vrn}) on real hyperbolic spaces and by Anker et al. \cite{AE}, which proves analogue of (\ref{vrn}) for all Riemannian symmetric spaces of noncompact type.
Precisely, it was proved in \cite{Vaz-2, AE} that if $f\in L^1(X)$ is left $K$-invariant then
\begin{equation}\label{vav-1}
\lim_{t\to\infty}\|f\ast h_t-\what f(i\rho)h_t\|_{L^1(X)}=0,
\end{equation}
with $h_t$ being the heat kernel of the Riemannian manifold $X$.
Vázquez's argument in \cite{Vaz-2} uses $L^1$ concentration of $h_t$ for real hyperbolic spaces which was later generalized in \cite{AE}. However, in \cite{AE}, another proof of this result was given which uses the theory of spherical Fourier transform in a significant way. This Fourier analytic proof from \cite{AE} is of fundamental importance to us.

In \cite {Vaz-2, AE}, it has been explained in great detail why (\ref{vav-1}) does not hold true for functions $f$ defined on $X$ which are not left $K$-invariant. Hence, in this paper, we shall  restrict our discussion mostly to functions $f:G\rightarrow\C$ which are $K$-biinvariant.

In the above mentioned papers it has also been shown that the exact analogue of (\ref{vrn}) breaks down for $p=\infty$. However, for the case $p=\infty$, a weak analogue of \eqref{vrn} holds in the sense that for all $f\in L^1(X)$,
\begin{equation}\label{infcase}
\|f\ast h_t-\what f(i\rho)h_t\|_{L^{\infty}(X)}\leq 2\|f\|_{L^1(X)}\|h_t\|_{L^{\infty}(X)}.
\end{equation}
Because of this weak $L^{\infty}$ case, both the papers \cite{Vaz-2, AE} could only prove a weak analogue of (\ref{vrn}) of the following kind:
for all $p\in (1,\infty)$,
\begin{equation}\label{vav}
\lim_{t\to\infty}\frac{\|f\ast h_t-\what f(i\rho)h_t\|_{L^p(X)}} {\|h_t\|_{L^{\infty}(X)}^{1/p'}}=0.
\end{equation}
From the estimates of $\|h_t\|_{L^p(X)}$, given in \cite[p. 1070]{AL} (see also \eqref{ehe-1}), it is clear that
\begin{equation}\label{vp}
\lim_{t\to\infty}\frac{\|h_t\|_{L^p(X)}}{\|h_t\|_{L^{\infty}(X)}^{1/p'}}=0,\:\:\:\:p\in (1,\infty).
\end{equation}
From \eqref{vp}, the relation \eqref{vav} follows simply by using the Young's inequality. In fact, the relation \eqref{vav} holds true even if $f$ is not left $K$-invariant and $\what f(i\rho )$ is replaced by any $z\in\C$. It is for this latter reason that we feel, for $p>1$, the result \eqref{vav} is not an exact analogue of \eqref{vrn}. 

Our viewpoint is that to obtain an exact analogue of (\ref{vrn}) it is essential to be able to replace the quantity $\|h_t\|_{L^{\infty}(X)}^{1/p'}$ in the denominator of (\ref{vav}) by $\|h_t\|_{L^p(X)}$. Theorem \ref{cl16}and Theorem \ref{fcl16} achieve that but at the cost of replacing $\what f(i\rho )$ by constants which depend on $f$ and $p$. The role of the constant can be easily seen for the case $p=2$. In this case, we will see that for all left $K$-invariant $f\in L^1(G/K)$ it follows from the Plancherel theorem that 
\begin{equation}\label{2}
\lim_{t\to\infty}\frac{\|f\ast h_t-\widehat{f}(0)h_t\|_{L^2(X)}}{\|h_t\|_{L^2(X)}}=0,
\end{equation}
(see Proposition \ref{p=2}). This is a significant departure from the Euclidean case. It also shows that in our main theorems we can allow $f$ to be in a much larger space of functions on $X$ than $L^1(X)$ (see (\ref{funspace})). We will also show that our $L^p$ versions have some bearing with the so called Herz criterion regarding norm of convolution operators on $L^p$ spaces (see Remark \ref{irem-1}).

To prove Theorem \ref{cl16} for the case $p\in (2,\infty]$, we shall use the Kunze--Stein phenomenon (see Theorem \ref{ksphn}). However, it can also be proved by an alternative method (see Remark \ref{withoutkz}). Using the case $p=2$ and the Kunze--Stein phenomenon, we shall obtain a genuine version of \eqref{vav} for $p\in (2,\infty]$ with $\what f(i\rho )$ replaced by $\what f (0)$. As has been mentioned before, we shall prove an  $L^p$-analogue of \eqref{vav-1} for the following class of left $K$-invariant functions on $X$, which contains all left $K$-invariant integrable functions on $X$. For $p \in [1, 2]$, we define
\begin{equation}\label{funspace}
\mathcal L^p(G//K)=\{ f : G \to \C   \text{ measurable and $K$-biinvariant} \mid \what{|f|}(i\gamma_p \rho )< \infty \}.
\end{equation}
Our motivation to consider these spaces of functions stems from the Herz criterion, which states that for a non-negative,  $K$-biinvariant, measurable function $f$ on $G$, the convolution operator $T_f(\psi)=f\ast \psi$, $\psi\in L^p(G//K)$, $p\in [1,2]$, is $L^p$ bounded if and only if
\begin{equation*}
\what{f}(i\gamma_p\rho )=\int_Gf(x)\varphi_{-i\gamma_p\rho}(x)\:dx<\infty,
\end{equation*}
and in this case the norm $|||T_f|||_p$ of the operator $T_f$ equals $\what f (i\gamma_p \rho)$ \cite[Theorem 7]{Her} (see also \cite{Cow-h}). In particular, for $f\in \mathcal L^p(G//K)$, $\psi\in L^p(G//K)$, $p\in [1,2]$, it follows that
\begin{equation}
\label{hqwe1}
\|f \ast \psi \|_{L^p(G)} \le \what{|f|}(i\gamma_p \rho ) \|\psi\|_{L^p(G)}.
\end{equation}
It also follows from the well-known estimates of the spherical functions (see \eqref{spest}) that $\displaystyle{L^q(G//K)\subset \mathcal L^p (G//K)}$, $1\leq q<p\leq 2$.

\begin{rem}
    \textup{After completing the work for this paper, we became aware of \cite{Eff3} which proves Theorem \ref{cl16} for a subclass of functions $f\in L^1(X)$ without the assumption of left $K$-invariance.}
\end{rem}

We now talk about another result of similar nature which was proved for $\R^n$ in \cite[Proposition 3.10]{Net}: if for $r\in(0,\infty)$, $\displaystyle{m_r=\frac{\chi_{B(0,r)}}{|B(0,r)|} }$, where $B(0,r)$ is the open ball of radius $r$ centered at the origin, then for any complex Borel measure $\mu$ on $\R^n$,
\begin{equation}\label{hn}
\lim_{r\to\infty}\|\mu\ast m_r-\mu(\R^n)m_r\|_{L^1(\R^n)}=0.
\end{equation}
We will demonstrate that the exact analogue of \eqref{hn} does not hold in the case of $X$. Specifically, we will show that when the rank of $X$ is one, for a left $K$-invariant, non-negative and nonzero $f \in L^1(X)$,
\begin{equation}
	\label{cer1}
	\displaystyle\lim_{r \to \infty}\|m_r\|_{L^p(X)}^{-1}\|f \ast m_r- z\,m_r\|_{L^p(X)}\neq 0,
\end{equation}
for all $p\in[1, \infty]$, and all $z \in \C$ (see Subsection \ref{ballc}).

The organization of this paper is as follows. In Section \ref{prelim}, we establish notation, and all required preliminaries are gathered here. Section \ref{hk561} deals mainly with the results related to the heat kernel including the proof of Theorem \ref{cl16}. We prove Theorem \ref{fcl16}  in Section \ref{sef1} and also show that Theorem \ref{cl16} and  Theorem \ref{fcl16} are not true without the assumption that $f$ is $K$-biinvariant. In Section \ref{opc}, we discuss optimality of some of our results (Proposition \ref{sness-1}) and prove \eqref{cer1} for rank one symmetric space of noncompact type.

\section{Preliminaries}\label{prelim}
In this section we shall establish notation and collect all the ingredients required for this paper.
\subsection{Basic Notations}  The letters  $\R,\, \R^+, \,\C$ and $\N$ denote respectively  the set of  real numbers, positive real numbers, complex numbers and natural numbers. For $z\in \C$,  $\Re z$ and $\Im z$  denote respectively the real and imaginary parts of $z$. For a real number $a$, let $a_+$ denote the non-negative part of $a$ defined by $a_+:= \max{\{a, 0\}}$. For $1 < p < \infty$, let $p':=\frac p{p-1}$ be the conjugate exponent of $p$. For $p=1$, we define $p'=\infty$ and for $p=\infty$, we define $p'=1$. For $1 \le p <  \infty$, let $\displaystyle{\gamma_p=\frac{2}{p}-1}$ and let $\gamma_\infty=-1$. For a measure space $Y$,  let $L^p(Y)$ denote the usual Lebesgue spaces over $Y$. We denote by $\|f\|_p$ the $L^p$ norm  of $f\in L^p(Y)$, when there is no risk of confusion about the underlying measure space. Everywhere in this article the symbol $f_1\asymp f_2$ for two positive expressions $f_1$ and $f_2$
means that there are positive constants $C_1, C_2$ such that $C_1f_1\leq f_2\leq C_2f_1$. Similarly the symbol $f_1\lesssim f_2$ (respectively $f_1\gtrsim f_2$)  means that there exists a positive constant $C$ such that $f_1\leq Cf_2$ (respectively $f_1\geq Cf_2$). We write $C_{\varepsilon}$ for a constant to indicate its dependency on the parameter $\varepsilon$.
For  a set $A$ in a topological space, $\overline{A}$ is  its closure and for a measurable set $A$ in a measure space, $|A|$ denotes its measure. For a topological space $E$,  let $C(E)$ and $C_c(E)$ denote the space of all continuous functions and compactly supported continuous functions on $E$ respectively.

\subsection{Symmetric space}\label{pe1}
Let $X$ be a Riemannian symmetric space of noncompact type. That is $X=G/K$ where $G$ is a noncompact connected semisimple Lie group with finite
center and $K$ is a maximal compact subgroup of $G$. Let $\mathfrak{g}$ and $\mathfrak{k}$ denote the Lie algebra of $G$ and $K$ respectively. Let $\Theta$ be the corresponding Cartan involution and $\mathfrak{g} =\mathfrak{k} \oplus \mathfrak {p}$ be the corresponding Cartan decomposition. We fix a maximal abelian subspace $\mathfrak{a}$ of $\mathfrak{p}$. Let $\ell$ denote the dimension of $\m a$, called the rank of $X$ as well as the real rank of $G$. Let $\Sigma$ denote the set of restricted roots of the pair $(\mathfrak{g, a})$ and $W$ be the Weyl group associated with $\Sigma$.  For $\alpha \in \Sigma$, let $\m g_{\alpha}$ be the root space corresponding to the root $\alpha$ and $m_\alpha$ be the dimension of $\m g_{\alpha}$. Fix a positive Weyl chamber $\mathfrak{a}_+ \subset \mathfrak{a}$. Let $\Sigma^+$ denote the corresponding set of positive roots and and $\rho= \frac 1 2\displaystyle\sum_{\alpha \in \Sigma^+} m_\alpha \alpha$. Let $\Sigma_r^+$ and $\Sigma_s^+$ denote the set of positive indivisible roots and positive simple roots respectively. Let
\begin{equation}
n=\ell+ \displaystyle\sum_{\alpha \in \Sigma^+} m_{\alpha},\hspace{0.5cm}\nu= \ell + 2|\Sigma_r^+|.
\end{equation}
The constant $n$ is equal to the dimension of the manifold $X$ and the constant $\nu$ will be referred to as the dimension at infinity.

 Let $\mathfrak{n}= \displaystyle\sum_{\alpha \in \Sigma^+} \mathfrak {g}_{\alpha}$ be  the nilpotent Lie subalgebra associated with $\Sigma^+$. Let $A$ and $N$ denote the analytic subgroups of $G$ with Lie algebras   $\m a$ and $\m n$ respectively, precisely $A= \exp\m a$ and $N= \exp \m n$. For $a \in A$, let $\log(a)$ be the element of $\m a$ such that $\exp (\log a)=a$. Let $M$ be the  centralizer of $A$ in $K$. The groups $M$ and $A$ normalize $N$. We note that $K/M$ is the Furstenberg boundary of $X$. For the group $G$, we  have the Iwasawa decomposition
$
G= NAK,
$
that is, every $g\in G$
can be uniquely written as
\[
g=n(g)\exp A(g)k(g), \,\,\,\, n(g)\in N, A(g)\in \mathfrak{a}, k(g)\in K,
\]
and the map
$
(n, a, k) \mapsto nak
$
is a  diffeomorphism of $N\times A \times K$ onto $G$.  Via the Cartan decomposition i.e. $G= K (\exp \overline{\m a_+})K$,
each $g\in G$ can be written as
\begin{equation}
\label{ps1}
g=k_1 (\exp H) k_2,
\end{equation}
for some $k_1, k_2 \in K$ and a unique $H \in \overline{\m a_+}$.
Let $g^+$ denote the middle component of $g$  in the Cartan decomposition i.e. $H$ in the decomposition \eqref{ps1}.

Let $B$ denote the Killing form of
$\m g$. It is known that $B|_{\m{p\times p}}$ is positive definite and hence induces an inner product and a
norm $|\cdot|$ on $\m p$. Since the tangent space of the homogeneous space $X= G/K$ at the base point $o=eK$ can be identified with $\m p$, it has a $G$-invariant Riemannian structure induced from the  form $B|_{\m{p\times p}}$. Let $\Delta$ be the Laplace--Beltrami operator on $X$ associated to this Riemannian structure. Using non-degeneracy of the positive bilinear form $B|_{\m {a \times a}}$, we identify $ \m a$ with $\m a^*$.
Let $\mathfrak{a}_\C$ denote the complexification of $\mathfrak{a}$.
Then $\mathfrak{a}_\C$ can be naturally identified with $\C^\ell$ and the inner product $B|_{\m {a \times a}}$ on $\m a$ can be extended  to a real bilinear  map on $\mathfrak{a}_\C$. We also extend the action of the Weyl group $W$ on $\m a_{\C}$ from $\m a$.

As usual on the compact group $K$ (respectively on $K/M$), we fix the normalized Haar measure $dk$ (respectively the normalized $K$-invariant measure $dk_M$) and let $dn$ denote a Haar measure on $N$. The following integral formulae describe the Haar measure of $G$ in terms of the  Iwasawa and Cartan decomposition.
For any $f\in C_c(G)$,
\begin{eqnarray}
\int_{G}{f(g)\,  dg} &= &\int_{K}\int_{A}\int_{N}f(nak)e^{-2 \langle \rho, \log(a)\rangle}\, dn \,da\,dk\\
&=&\int_{K}{\int_{\overline{\m a+}}{\int_{K}{f(k_1(\exp g^+)k_2)  J(g^+)\,dk_1\,dg^+\,dk_2}}}, \label{ch1}
\end{eqnarray}
where $dg^+$ is the Lebesgue measure on $\R^\ell$ and
\begin{equation}
J(g^+)
 \asymp \prod_{\alpha\in \Sigma^+}\left( \frac{\langle \alpha, g^+ \rangle}{1+\langle \alpha, g^+ \rangle }\right)^{m_\alpha}
 e^{2\langle \rho, g^+\rangle}, \,\,\,\, \,g^+ \in \overline{\mathfrak{a}_+}.
 \label{ch2}
\end{equation}
For two suitable functions $f_1$ and $f_2$ on $G$, their convolution $f_1\ast f_2$ is defined as
\begin{equation*}
f_1\ast f_2(g)=\int_Gf_1(g_1)f_2(g_1^{-1}g)\:dg_1,\:\:\:\:\:g\in G.
\end{equation*}
An important result regarding convolutions on $G$ is the Kunze--Stein phenomenon:
\begin{thm}\cite[Theorem 2.2]{Cow1}\label{ksphn}
Suppose $G$ is a connected, noncompact, semisimple Lie group with finite center. Then for all $\phi\in L^p(G)$, $\psi\in L^2(G)$, where $p\in[1,2)$, there exists a constant $c_p>0$ such that
$$\|\phi\ast\psi\|_{L^2(G)}\leq c_p\|\phi\|_{L^p(G)}\|\psi\|_{L^2(G)}.$$
It then follows by duality that for all $f_j\in L^2(G)$, $j=1,2$ and for any $p\in (2,\infty]$, there exists a constant $C_p>0$ such that
\begin{equation}\label{cks}
\|f_1\ast f_2\|_{L^p(G)}\leq C_p \|f_1\|_{L^2(G)}\|f_2\|_{L^2(G)}.\end{equation}
\end{thm}
We will identify a  function  on $X= G/K$ with a right $K$-invariant function on $G$. For an integrable function $f$ on $X$, \[\int_Gf(g)\,dg=\int_Xf(x)\,dx,\] where in the left hand side $f$ is considered as a right $K$-invariant function on $G$ and $dg$ is the Haar measure on $G$, while on the right hand side $dx$ is the  $G$-invariant measure on $X$. We shall slur over the difference between integrating over $G$ and that on $X=G/K$ as we shall deal with functions on $X$.
For a function space $\mathcal F(X)$ on $X$, let $\mathcal F(G//K)$ denote the subspace of $\mathcal F(X)$ consisting of  left $K$-invariant functions in $\mathcal F(X)$. Similarly, let $\mathcal F(X)^W$ denote the space of all $W$-invariant functions in $\mathcal F(X)$.
\subsection{Spherical Fourier analysis}
For $\lambda \in \m a_{\C}$, the elementary spherical function $\varphi_\lambda$ is defined by
\begin{equation} \label{esf}
 \varphi_\lambda(x)= \int_{K} e^{\langle i\lambda+\rho, A(k^{-1}x) \rangle}\,dk, \:\:\:\:x \in G.
\end{equation}
The elementary spherical function $\varphi_\lambda$ has the following properties (\cite{Helga-2, Helga-3}).
\begin{itemize}
\item [(a)] $\varphi_\lambda(x)$ is a smooth, $K$-biinvariant in $x$, and Weyl group invariant function in $\la$. Precisely, $\varphi_{\lambda_1}$ and  $\varphi_{\lambda_2}$ are identical if and only if $\lambda_1 = w \cdot\lambda_2$ for some $w \in W$.
\item[(b)] For all $H \in \overline{ \m a_+}$, the spherical functions $\varphi_{-i\gamma_p\rho}$, $p\in [1,2]$, satisfy the following estimate (\cite[p. 161]{GV}, \cite{A}) :
\begin{equation}\label{spest}
\begin{rcases}
    \varphi_{-i\gamma_p\rho}(\exp H)&\lesssim& (\Pi_{\alpha \in \Sigma_r^+}(1+ \langle \alpha, H \rangle)) e^{-\frac{2}{p'}\langle \rho, H \rangle},\:\:\:p<2\\
    &\asymp& (\Pi_{\alpha \in \Sigma_r^+}(1+ \langle \alpha, H \rangle)) e^{-\langle \rho, H \rangle},\:\:\:p=2.
\end{rcases}
\end{equation}
In particular, $\varphi_{-i\gamma_p\rho}\in L^q(G//K)$ for all $q\in (p',\infty]$.
\item [(c)] For each fixed $\lambda \in \m a_\C$,
\begin{equation*}
\Delta \varphi_\lambda = -(|\lambda|^2 + | \rho|^2)\varphi_\lambda,\:\:\:\:\varphi_\lambda(e)=1,
\end{equation*}
with $\varphi_\lambda(x)= \varphi_{-\lambda}(x^{-1})$, for all $x\in G$.
\item [(d)] For each $x \in G$, $\varphi_\lambda(x)$ is holomorphic in $\lambda \in \m a_{\C}.$
\end{itemize}
For $\lambda \in \m a_\C$, and a suitable function $f$ on $X$, the spherical Fourier transform of $f$ at $\lambda$ is defined as
\begin{equation}
\label{db1}
\what f(\lambda)= \int_G f(x)\varphi_{-\lambda}(x)\, dx.
\end{equation}
For $\lambda \in \m a_\C$, $kM \in K/M$ and a suitable function $f$ on $X$, the Fourier transform of $f$ at $(\lambda, kM)$ is given by
\begin{equation}
\label{db1h}
\wtilde f(\lambda, kM)= \int_X f(x) e^{\langle -i\lambda+\rho, A(k^{-1}x) \rangle} \, dx.
\end{equation}
It is known that if $f$ is left $K$-invariant function on $X$ then
\begin{equation*}
\wtilde f(\lambda, kM)=\what f(\lambda),\:\:\:\:\:kM\in K/M.
\end{equation*}
We shall need the following inequality regarding the Fourier transform $\tilde{f}$ (see \cite[p.209]{Helga-2}), which easily follows from the fact that $\|\varphi_{\la}\|_{\infty}\leq 1$, if and only if $|\Im\la|\leq|\rho|$
\begin{equation}
\int_{K/M}|\tilde{f}(\la,kM)|dk_M\leq \|f\|_1,\:\:\:\text{for all $f\in L^1(X)$, $|\Im\la|\leq |\rho|$.}\label{helgs2}
\end{equation}
For two suitable functions $f_1$ and $f_2$ on $X$, their convolution $f_1\ast f_2$ is defined as
\begin{equation*}
f_1\ast f_2(gK)=\int_Gf_1(g_1)f_2(g_1^{-1}g)\:dg_1,\:\:\:\:\:g\in G.
\end{equation*}
Moreover, if $f_2$ is left $K$-invariant then
\begin{equation*}
	\wtilde{f_1\ast f_2}(\la,kM)=\wtilde{f_1}(\la,kM)\what{f_2}(\la).
\end{equation*}
We shall also need the following formula regarding the Fourier transform of translates of functions (see \cite[p.200]{Helga-2} ): for $f\in L^1(G//K)$ and $y\in G\setminus K$ if we consider the left translate $l_yf(x)=f(y^{-1}x)$, $x\in G$, then
\begin{equation*}
\wtilde{l_yf}(\la,kM)=e^{\langle -i\la+\rho,A(k^{-1}y)\rangle}\what f (\la),\:\:\:\:\:k\in K,\:\: |\Im\la|\leq |\rho|.\label{helgs1}
\end{equation*}
It is well known that for $f\in L^1(G//K)$ and $|\Im\la|\leq |\rho|$
\begin{equation}\label{spproj}
	f\ast\varphi_{\lambda}(x)=\what f(\la)\varphi_{\lambda}(x),\:\:\:\:\:x\in X.
\end{equation}
From the estimate \eqref{spest} it can be easily seen that for $f\in L^p(G//K)$, $p\in [1,2)$ the integral \eqref{db1} is convergent. Moreover, if $f\in L^p(G//K)\cap C(G//K)$, $p\in [1,2)$ and $\what {f}\in L^1(\m a, |\hc(\lambda)|^{-2}\, d\lambda)$ then the following pointwise Fourier inversion holds (\cite[Theorem 3.3]{ST}):
\begin{equation}
\label{inv-1}
f(x)= C_o \int_{\m a} \what f(\lambda) \varphi_\lambda(x)|\hc(\lambda)|^{-2}\, d\lambda, \:\:\:\:\:\text{for all $x \in G$},
\end{equation}
where $C_o= 2^{n-\ell}/(2\pi)^\ell|K/M||W|$ and $\hc(\lambda)$ is Harish--Chandra's $\hc$-function (see \eqref{cfn-1} below). It is also known  that the spherical Fourier transform $\what f$ defined initially on $C_c^{\infty}(G//K)$ by the formula \eqref{db1}, extends to an
isometry of $L^2(G//K)$ onto $L^2(\m a, |\hc(\lambda)|^{-2}\, d\lambda)^W$ such that the following Plancheral formula holds:
\begin{equation}
\label{pla-1}
\int_X |f(x)|^2 \, dx= \int_{\m a} |\what {f}(\lambda)|^2  |\hc(\lambda)|^{-2}\, d\lambda.
\end{equation}
We shall need the following version of Hausdorff--Young inequality for the case of rank one.
\begin{thm}\cite[Theorem 2.1]{Cow1}
For $p\in(1,2)$, there is a positive constant $C_p$ such that

\begin{equation}\label{hyineq}
\int_{0}^{\infty}|\what F(\la-i\gamma_p\rho)|^{p'}\tilde{\mu}(\la)\:d\la \leq C_p\|F\|^{p'}_p,
\end{equation}
for all $F\in L^p(G//K)$, where $\tilde{\mu}(\la)=(1+|\la|)^{m_{\alpha}+m_{2\alpha}},\:\:\alpha\in\Sigma_r^+$.
\end{thm}

For $\alpha \in \Sigma_r^+$, $z \in \C$ with $\Im z<0$, let
\begin{equation*}
\hc_\alpha(z)=v(\alpha)\frac{\Gamma(iz)}{\Gamma\left(iz+\frac 12 m_\alpha\right)}\frac{\Gamma(\frac i 2 z+\frac{1}{4}m_\alpha)}{\Gamma\left (\frac i2 z+\frac 14 m_\alpha+ \frac 12 m_{2\alpha}\right)},
\end{equation*}
where
\begin{equation*}
v(\alpha)=\frac{\Gamma\left(\frac{\langle \alpha, \rho \rangle}{\langle \alpha, \alpha \rangle}+\frac 12 m_\alpha\right)}{\Gamma\left(\frac{\langle \alpha, \rho \rangle}{\langle \alpha, \alpha \rangle}\right)}
	\frac{\Gamma\left(\frac 12\frac{\langle \alpha, \rho \rangle}{\langle \alpha, \alpha \rangle}+\frac 14 m_\alpha+\frac 12 m_{2\alpha}\right)}{\Gamma\left(\frac 12\frac{\langle \alpha, \rho \rangle}{\langle \alpha, \alpha \rangle}+\frac 14 m_\alpha\right)}.
\end{equation*}
The following formula for $\hc$-function is due to Gindikin--Karpelevi\v{c} (\cite[Theorem 6.13]{Helga-3}:
\begin{equation}
\label{cfn-1}
\hc(\lambda)= \prod_{\alpha \in \Sigma_r^+}\hc_\alpha\left( \frac{\langle \alpha, \lambda \rangle}{\langle \alpha, \alpha \rangle}\right).
\end{equation}
It is known that
\begin{equation}
|\hc(\lambda)|^{-2}= \hc(w\cdot\lambda)^{-1}\hc(-w\cdot\lambda)^{-1},\:\:\:\:\lambda \in \m a,\:\: w \in W.
\end{equation}
Following \cite{AE} we define
\begin{equation}
\label{ltsq1}
\hb(\lambda)= \pi(i\lambda)\hc(\lambda),
\end{equation}
where
\begin{equation}
\label{cgeq1}
\,\,\,\pi(\lambda)= \displaystyle \prod_{\alpha \in \Sigma_r^+} \langle \alpha, \lambda \rangle.
\end{equation}
From \eqref{cfn-1} it can be seen that the $\hc$-function is holomorphic in $\{\lambda\in \mathfrak a_{\C}\mid -\Im \lambda\in\mathfrak a_+\}$ and hence the function $\lambda \mapsto \hb(-\lambda)^{\pm 1}$ is holomorphic in a neighborhood of $\m a +i \overline{\m a^+}$ which satisfies the following estimate:
 \begin{equation}
 \label{hb1}
 |\hb(-\lambda)|^{\pm 1 } \asymp  \displaystyle \prod_{\alpha \in \Sigma_r^+} (1+|\langle \alpha, \lambda \rangle|)^{\mp\frac{m_\alpha+m_{2\alpha}}2 \pm 1},\:\:\:\:\:\:\lambda \in \m a +i \overline{\m a_+},
 \end{equation}
and whose derivatives satisfy the estimates:
\begin{equation}
\label{hb2}
|p(\nabla_\lambda)\hb(-\lambda)|^{\pm 1 }=O(|\hb(-\lambda)|^{\pm 1 } ),\:\:\:\:\:\:\lambda \in \m a +i \overline{\m a_+},
\end{equation}
where
$\nabla_\lambda=(\frac {\partial}{\partial \lambda_1}, \frac {\partial}{\partial \lambda_2}, \cdots, \frac {\partial}{\partial \lambda_{\ell}})$ and $p(\nabla_\lambda)$ is any differential polynomial (for details see \cite[p.7-8]{AE}).
From \eqref{ltsq1} and \eqref{hb1}, it is clear that  the $\hc$-function has the following global behaviour:
\begin{equation}
\label{cfest}
|\hc(\lambda)|\asymp \prod_{\alpha \in \Sigma_r^+}| \langle \alpha, \lambda \rangle|^{-1}(1+| \langle \alpha, \lambda \rangle|)^{1-\frac{m_\alpha+m_{2\alpha}}2}, \,\,\,\,\,\,\Im \lambda \in -\overline{\m a_{+}}.
\end{equation}

We shall also need Harish--Chandra's expansion of spherical functions. A vector $\lambda \in \m a$ is said to be regular if $\langle \lambda , \alpha  \rangle $ is nonzero for every $\alpha \in \Sigma$. For all regular $\lambda \in \mathfrak{a}$, the following convergent Harish--Chandra expansion is known to hold
\begin{equation}
\label{es1}
\varphi_\lambda(\exp H)= \sum_{w \in W} \hc(w. \lambda) \Phi_{w.\lambda}(H),
\end{equation}
for all $H \in \liea_+$, where
\begin{equation}
\label{es2}
\Phi_\lambda(H)=e^{\langle i\lambda-\rho, H\rangle} \displaystyle \sum_{q \in 2Q}\Gamma_q(\lambda)e^{-\langle q, H\rangle},\:\:\:\:\:Q= \sum_{\alpha \in {\Sigma_s^+}} \N \alpha,
\end{equation}
with $\Gamma_0(\lambda)\equiv 1$. The other coefficients $\Gamma_q(\lambda)$ are rational functions in $\lambda \in \liea_{\C}$, which
have no poles in $\liea + i\overline{\liea_+}$ and satisfy the estimate
\begin{equation}
\label{cfs1}
|\Gamma_q(\lambda)| \le C (1+|q|)^d,\:\:\:\:\:\:\lambda \in \liea + i\overline{\liea_+}
\end{equation}
for some non-negative constants $C$ and $d$ independent of $q$ and $\lambda$. Moreover, for any differential polynomial $p(\nabla_\lambda)$, we have the estimate
\begin{equation}
\label{cfs2}
|p(\nabla_\lambda)(\Gamma_q(\lambda))|\lesssim |\Gamma_q(\lambda)|,\:\:\:\:\:\:\lambda \in \liea + i\overline{\liea_+}.
\end{equation}
\subsection{Fractional heat kernel}
For $\alpha \in (0, 1)$, there exists a function $\eta_t^\alpha:\R^+\rightarrow\R^+$ (known as the  Bochner's subordinator function) such that the following  subordination formula holds \cite[p. 88]{Grac}:
\begin{equation}
\label{fhss2}
e^{-ts^{\alpha}}= \int_{0}^\infty  e^{-su} \eta_t^{\alpha}(u)\, du,\,\,\,\, \,s \ge 0.
\end{equation}
The subordinator function $\eta^\alpha_t$ satisfies the following global estimate \cite[(8)--(9)]{Grac}:
\begin{equation}
\label{fehe-2}
   \eta_t^\alpha(u) \asymp \left\{\begin{array}{lr}
        t^{\frac{1}{2(1-\alpha)}}u^{-\frac{(2-\alpha)}{2(1-\alpha)}}e^{-c_\alpha t^{\frac{1}{1-\alpha}}u^{-\frac{\alpha}{1-\alpha}} }, & u \le t ^{ \frac 1 \alpha}, \\
        t u^{-(1+\alpha)}e^{-c_\alpha t^{\frac{1}{1-\alpha}}u^{-\frac{\alpha}{1-\alpha}} }, & u > t ^{ \frac 1 \alpha},
 \end{array}\right.
\end{equation}
where $c_\alpha=(1-\alpha) \alpha^{\frac \alpha{1-\alpha}}.$ From \eqref{fehe-2} it follows that
there exist real numbers $a_1$ and $a_2$ such that for all $t\in (1,\infty)$, $u\in (0,\infty)$
\begin{equation}
\label{fehe-3}
   \eta_t^\alpha(u) \lesssim
        t^{a_1}(1+ u^{a_2})e^{-c_\alpha t^{\frac{1}{1-\alpha}}u^{-\frac{\alpha}{1-\alpha}} }.
 \end{equation}

For $\alpha\in(0,1)$, the fractional heat kernel $h_t^\alpha$ is defined as
\begin{equation}
\label{fhss1}
h_t^\alpha(x)= \int_{0}^\infty  h_u(x)   \eta_t^{\alpha}(u)\, du,\,\,\,\,\,t>0, x\in X.
\end{equation}

It is known that $h_t^\alpha$ is a positive $K$-biinvariant $L^2$-function on $X$ whose spherical Fourier transform is given by
\begin{equation}
h_t^\alpha(\lambda)=e^{-t(|\lambda|^2+|\rho|^2)^\alpha}, \,\,\,\,\,\lambda \in \m a.
\end{equation}
We observe that $h_t^{\frac 12}$ is the Poisson kernel of $X$.
Let  $u$ be an $L^2$-tempered distribution on $X$.  We define the action of $(-\Delta)^\alpha$ on $u$ by
$$\widehat{(-\Delta)^{\alpha} u }= -\left(|\cdot|^2+|\rho|^2\right)^{\alpha} \what u.$$
It is known that  $h^\alpha_t$ is the fundamental solution of the fractional heat equation on $X$:
\[\partial_t u+(-\Delta)^\alpha u=0, \,\, \, t>0.\]
It is also known that for every $t>0$ and $s>0$
\begin{equation*}
	\int_X h^\alpha_t(x)\, dx=1,\hspace{1cm}
	h_t^\alpha\ast h_{s}^\alpha= h_{t+s}^\alpha.
\end{equation*}
Norm estimates similar to \eqref{ehe-1} are also available for $h_t^{\alpha}$. For all $t\in (1,\infty)$ and $\alpha \in (0,1)$, we have the following norm estimates of $h_t^{\alpha}$.
\begin{equation}
\label{fehe-1}
   \|h_t^\alpha\|_p \asymp \left\{\begin{array}{lr}
        t^{-\frac{\ell}{2p'}}e^{-t \left(\frac{4|\rho|^2}{pp'}\right)^\alpha }, & 1 \le p<2,\\
        t^{-\frac{\nu}{4}}e^{-{|\rho|^{2\alpha} t}}, &  p=2,\\
        t^{-\frac{\nu}{2}}e^{-{|\rho|^{2\alpha} t}}, &  2< p\le \infty.
 \end{array}\right.
\end{equation}
The estimates for $p \in [1,2]$ and $p=\infty$ are provided in \cite[Lemma 3 (i)--(iii)]{Cow2}. For the remaining range $p \in (2,\infty)$, it is enough to establish that
\begin{equation}
\label{ep10}
{\|h_t^{\alpha}\|_p}\asymp \|h_{t/2}^\alpha\|_2^2,\:\:\:\:\text{for all $t>1$}.
\end{equation}
Applying the H$\ddot{\text{o}}$lder's inequality, we obtain
$$\|h_{t+1}^\alpha\|_{\infty}=\|h_1^{\alpha}\ast h_t^{\alpha}\|_{\infty}\leq\|h_1^{\alpha}\|_{p^{\prime}}\|h_t^{\alpha}\|_{p}\lesssim \|h_t^{\alpha}\|_{p}.$$
Using the estimate of $\|h_t^{\alpha}\|_p$ for $p=\infty$, and $p=2$, we get
$$\|h_t^{\alpha}\|_{p}\gtrsim (t+1)^{-\frac{\nu}{2}}e^{-(t+1)|\rho|^{\alpha}}\gtrsim t^{-\frac{\nu}{2}}e^{-t|\rho|^{\alpha}}\gtrsim\|h_{t/2}^\alpha\|_2^2.$$
The reverse inequality follows straight away from \eqref{cks}
$$\|h_t^{\alpha}\|_{p}=\|h_{t/2}^{\alpha}\ast h_{t/2}^{\alpha}\|_{p}\lesssim\|h_{t/2}^{\alpha}\|_2^2.$$
\section{Results for the heat  kernel}
\label{hk561}
We start with the case $p=2$. As has been mentioned in the introduction, the following result and \eqref{vav-1} serves as the motivation for formulation of the statements of our main results.
\begin{proposition}
\label{p=2}
If $f\in L^1(G//K)$ then
\begin{equation*}
\lim_{t \to \infty}\|h_t\|_2^{-1}\|f \ast h_t-\what f(0)h_t\|_2=0.
\end{equation*}
Moreover, for any $z \in \C$,
\begin{equation*}
	\lim_{t \to \infty}\|h_t\|_2^{-1}\|f \ast h_t- z\,  h_t\|_2=0,\:\:\:\text{ if and only if } \:\:z= \what f(0).
\end{equation*}
\end{proposition}
\begin{proof}
Replacing $\|h_t\|_2$ by the estimate \eqref{ehe-1} and using the Plancherel formula \eqref{pla-1} we obtain
\begin{eqnarray}
&&\|h_t\|_2^{-2}\|f \ast h_t-z\, h_t\|_2^2 \no \\
&\asymp& t^{\frac \nu 2}  e^{2t|\rho|^2}  \int_{\m a} |\what {f}(\lambda)-z|^2 e^{-2t(|\lambda|^2+|\rho|^2)}  |\hc(\lambda)|^{-2}\, d\lambda \no \\
&\asymp& t^{\frac \nu 2} \int_{\m a} |\what {f}(\lambda)-z|^2 e^{-2t|\lambda|^2} \left[ \prod_{\alpha \in \Sigma_r^+}| \langle \alpha, \lambda \rangle|^{2}(1+| \langle \alpha, \lambda \rangle|)^{{m_\alpha+m_{2\alpha}}-2}\right] \, d\lambda \no \\
&\asymp & \int_{\m a} \left|\what {f}\left(\frac{\lambda}{\sqrt t}\right)-z\right|^2 e^{-2|\lambda|^2} \left[ \prod_{\alpha \in \Sigma_r^+}| \langle \alpha, \lambda \rangle|^{2}\left(1+\frac{| \langle \alpha, \lambda \rangle|}{\sqrt t}\right)^{({m_\alpha+m_{2\alpha}}-2)}\right]d\lambda,
\label{flhm2}
\end{eqnarray}
where we used the estimate \eqref{cfest} of $|\hc(\lambda)|$ in the third line. Since $f \in L^1(G//K)$, we observe that
\begin{eqnarray}
&\,&\left|\what {f}\left(\frac{\lambda}{\sqrt t}\right)- z \right|^2 \left[ \prod_{\alpha \in \Sigma_r^+}| \langle \alpha, \lambda \rangle|^{2}\left(1+\frac{| \langle \alpha, \lambda \rangle|}{\sqrt t}\right)^{({m_\alpha+m_{2\alpha}}-2)}\right] \no \\
 &\lesssim &\prod_{\alpha \in \Sigma_r^+}| \langle \alpha, \lambda \rangle|^{2}\left(1+ {| \langle \alpha, \lambda \rangle|}\right)^{({m_\alpha+m_{2\alpha}}-2)_+},\:\:\:\:\:\text{for all $t >1$.}
 \label{flhm1}
 \end{eqnarray}
 From \eqref{flhm1} and the dominated convergence theorem it follows from \eqref{flhm2} that
 \begin{eqnarray}
 \lim_{t \to \infty}\|h_t\|_2^{-2}\|f \ast h_t-z\, h_t\|_2^2
 \asymp  \int_{\m a} \left|\what {f}(0)-z\right|^2 e^{-2|\lambda|^2} \left( \prod_{\alpha \in \Sigma_r^+}| \langle \alpha, \lambda \rangle|^{2}\right)\, d\lambda.
 \label{flmh3}
 \end{eqnarray}
From \eqref{flmh3} it is clear that
\begin{equation*}
\lim_{t \to \infty}\|h_t\|_2^{-1}\|f \ast h_t- z\,  h_t\|_2=0,
\end{equation*}
if and only if $z= \what f(0)$.
\end{proof}
We now proceed towards the case $p\in [1,2)$. The strategy is to first prove an $L^p$-analogue of \cite[Lemma 2.1]{AE} using \cite[Theorem 4.1.2]{AL}. In this regard, we consider a function  $r:\R^+\rightarrow \R^+$ satisfying the conditions
\begin{equation}\label{r}
\lim_{t \to \infty}\frac{r(t)}{\sqrt t} = \infty,\:\:\:\:\:\lim_{t \to \infty}\frac{r(t)}{ t}=0.
\end{equation}
Let $\Omega_t^p:= B(2t\gamma_p \rho, r(t))$ be the ball in $\m a$ with  radius $r(t)$ and  center $2t \gamma_p \rho$. The following proposition, which will be used later for proving Theorem \ref{cl16}(a), says that for $p\in [1,2)$  the $L^p$-norm of $h_t$ concentrates asymptotically in the set $K(\exp \Omega_t^p)K$. This set will be referred as the critical region for $L^p$-heat concentration for $p\in [1,2)$.
\begin{proposition}
\label{cl5}
If $p\in (1,2)$ then for any $N \ge 0$ and sufficiently large $t$ the following estimate holds:
\begin{equation*}
\|h_t\|_p^{-1} \|h_t \|_{L^p(G\setminus K (\exp \Omega_t^p) K)} \lesssim \left(\frac{r(t)}{\sqrt t}\right)^{-N}.
\end{equation*}
In particular,
\begin{equation*}
\lim_{t \to \infty}\|h_t\|_p^{-1} \|h_t \|_{L^p(G\setminus K (\exp \Omega_t^p) K)}=0.
\end{equation*}
\end{proposition}
\begin{proof}
The basic idea of the proof is similar to that of \cite[Lemma 2.1]{AE}. By \eqref{r} there exists $t_o\in(1,\infty)$ such that
\begin{equation}\label{deltat}
r(t) < \gamma_p|\rho| t,\:\:\:\text{for all}\:\: t > t_o.
\end{equation}
We define \begin{equation}\label{gtexph}
	g_t(H)= e^{-|\rho|^2t-\langle \rho, H\rangle-\frac{|H|^2}{4t}},\:\:\:\:\:\:H \in  \overline{\m a_+}.
\end{equation}
From the pointwise estimate of $h_t$ (see \eqref{hes-1}) it follows that for all $H \in  \overline{\m a_+}$
\begin{equation}
\label{ehe-45}
h_t(\exp H)^p \,\, {\lesssim g_t(H)}^p \times \left\{\begin{array}{ll}
        t^{-\frac{p\ell}{2}} & \text{if  } H\in B(0,3t\gamma_p|\rho|),\\

       |H|^{pL},&  \text{if  } |H| \ge 3t\gamma_p |\rho|,
        \end{array}\right.
\end{equation}
for all $t\in(1,\infty)$, where $L=\sum_{\alpha\in\Sigma_r^+}(m_{\alpha}+m_{2\alpha})/2$. We observe that
 \begin{eqnarray}
g_t(H)^pe^{2\langle \rho, H\rangle} & =& e^{-p|\rho|^2t+(2-p)\langle \rho, H\rangle-\frac{p|H|^2}{4t}}
\no \\
& =& e^{-\frac p{4t}\left({|H|^2}-4t\gamma_p\langle \rho, H\rangle+4|\rho|^2t^2\right)}\no \\
& =& e^{-\frac p{4t}|H-2t \gamma_p \rho|^2}e^{- \frac{4|\rho|^2t}{p'}} . \label{bs31}
 \end{eqnarray}
From \eqref{bs31} and the estimate of $\|h_t\|_p$ given in \eqref{ehe-1} we get for all $t\in(1,\infty)$
\begin{equation}
\label{bs322}
\|h_t\|_p^{-p}g_t(H)^pe^{2\langle \rho, H\rangle} \asymp t^{\frac{p\ell}{2p'}} e^{-\frac p{4t}|H-2t \gamma_p \rho|^2}.
\end{equation}
As $\Omega_t^p=B(2t\gamma_p\rho,r(t))\subset B(0,3t\gamma_p|\rho|)$ (see \eqref{deltat}), we note that
 \begin{align}
G\setminus K (\exp \Omega_t^p) K &=[K (\exp B(0, 3t\gamma_p|\rho|))K \setminus K (\exp \Omega_t^p) K]\no \\& \hspace{0.5cm}\sqcup [G \setminus K (\exp B(0, 3t\gamma_p|\rho|)) K]\label{setdecom}.
\end{align}
Therefore, it suffices to estimate the integrals of $\|h_t\|^{-p}h_t^p$ over each of these two sets.
From \eqref{ch2} it follows that the Jacobian $J(H)$ satisfies the estimate
\begin{equation*}
J(H) \lesssim e^{2\langle \rho, H\rangle},\hspace{0.2cm} \text{for}\:\:H \in \overline {\m a_{+}}.
\end{equation*}
Estimates \eqref{ehe-45} and \eqref{bs322} now implies that
\begin{eqnarray}
\label{ms1}
&&\no\|h_t\|_p^{-p} \int\limits_{K (\exp B(0, 3t\gamma_p|\rho|))K \setminus K (\exp \Omega_t^p) K}h_t(x)^p \, dx \\ \no
&\lesssim &  \|h_t\|_p^{-p} \displaystyle\int\limits_{ K (\exp B(0, 3t\gamma_p|\rho|))K \setminus K (\exp \Omega_t^p) K}
g_t(H)^p t^{-\frac{p\ell}{2}}J(H)\, dH\\ \no
&\lesssim & t^{-\frac{p\ell }{2}} t^{\frac{p\ell}{2p'}} \int\limits_{ |H-2t\gamma_p \rho |\geq r(t)}
e^{-\frac p{4t}|H-2t \gamma_p \rho|^2}\, dH\no\\
&\lesssim& t^{-\frac{\ell }{2}}  \int\limits_{ r(t)}^ {\infty}
e^{-\frac{ ps^2}{4t}}s^{\ell-1}\, ds=\int\limits_{ \frac {r(t)}{\sqrt t}}^{ \infty} e^{-\frac{ ps^2}{4}}s^{\ell-1}\, ds\lesssim
{ \left(\frac {r(t)}{\sqrt t}\right)}^{ -N},
\end{eqnarray}
for any $N >0$. Note that, to obtain the last inequality we have used the standard estimate of the incomplete gamma function:
\begin{equation*}
    \int_a^{\infty}s^{b-1}e^{-\frac{ps^2}{2}}\:ds=C_p\int_{\frac{pa^2}{2}}^{\infty}x^{\frac{b}{2}-1}e^{-x}\:dx\lesssim a^{b-1}e^{-\frac{pa^2}{2}},\:\:\:a>0,\:b>0.
\end{equation*}

Similarly, using \eqref{ehe-45} and \eqref{bs322} we obtain for all sufficiently large $t$
\begin{eqnarray}
\label{ms3}
&&\no \|h_t\|_p^{-p} \int\limits_{G \setminus K (\exp B(0, 3t\gamma_p|\rho|))K}h_t(x)^p \, dx\\\no &\lesssim &  \|h_t\|_p^{-p} \displaystyle\int\limits_{\overline{\m a_+} \setminus B(0, 3t\gamma_p|\rho|)\ }
g_t(H)^p |H|^{{ p L}}e^{2\langle \rho, H\rangle}\, dH\\ \no
&\lesssim & t^{\frac{p\ell}{2p'}} \int\limits_{|H| \ge  3t\gamma_p|\rho|} |H|^{{ p L}}
e^{-\frac p{4t}|H-2t \gamma_p \rho|^2}\, dH\\ \no
\no
&\lesssim & t^{\frac{p\ell}{2p'}} \int\limits_{|H-2 t\gamma_p\rho | \ge  \gamma_p|\rho|t}
|H|^{{ p L}}e^{-\frac p{4t}|H-2t \gamma_p \rho|^2}\, dH\\ \no
\no
&= & t^{\frac{p\ell}{2p'}} \int\limits_{|H| \ge  \gamma_p|\rho|t} |H+2 t\gamma_p \rho |^{{ p L}}
e^{-\frac p{4t}|H|^2}\, dH\\ \no
&\lesssim & t^{\frac{p\ell}{2p'}} \int\limits_{|H| \ge  \gamma_p|\rho|t} |H|^{{ p L}}
e^{-\frac p{4t}|H|^2}\, dH \hspace{2cm}(\text{as $|H+2 t\gamma_p \rho| \lesssim |H|$})\\ \no
&\lesssim & t^{\frac{p\ell}{2p'}} \int\limits_{ \gamma_p|\rho|t}^\infty s^{{ p L}}
e^{-\frac {ps^2}{4t}} s^{\ell-1}\, ds \\
&= & t^{\frac{p\ell}{2p'}+\frac \ell 2+\frac{pL}2} \int\limits_{ \gamma_p|\rho|\sqrt t}^\infty
e^{-\frac {ps^2}{4}} s^{pL+\ell-1}\, ds\no\\
&\leq&t^{\frac{p\ell}{2p'}+\frac \ell 2+\frac{pL}2} \int\limits_{\frac {r(t)}{\sqrt t} }^\infty
e^{-\frac {ps^2}{4}} s^{pL+\ell-1}\, ds\lesssim { \left(\frac {r(t)}{\sqrt t}\right)}^{ -N},
\end{eqnarray}
for any $N >0$, where we have used \eqref{deltat} to get the second last inequality and estimate of the incomplete gamma function to get the last inequality. The result now follows from \eqref{setdecom} and the estimates
\eqref{ms1} and \eqref{ms3}.\end{proof}

 For a vector $ v \in \liea$ and a smooth function $F$ on $\liea$, we denote $\partial_v F$ the directional derivative of $F$ in the direction of $v$ and for a set $E$, let $\mathcal{P}(E)$ denotes its power set. The following proposition, whose proof closely follows that of \cite[Proposition 3.1]{AE}, is the main ingredient for the proof of Theorem \ref{cl16}.
\begin{proposition}
\label{cl10}
If $p\in (1,2)$ and $f \in C_c^\infty(G//K)$ then
\begin{equation}
\lim_{t \to \infty}\|h_t\|_p ^{-1}\|f \ast h_t -\what f(i \gamma_p \rho)h_t\|_{L^p(K(\exp \Omega_t^p) K)}=0.
\end{equation}
\end{proposition}

\begin{proof}
We set \begin{equation}\label{ulam}
	U(\lambda)=\what f(\lambda)-\what f(i\gamma_p \rho),\:\:\:\:\:\:\la\in\mathfrak{a}.
\end{equation}
Based on the approach demonstrated in \cite[p.12-14]{AE}, we can similarly obtain for $H \in \m a$
\begin{equation}\label{lts-1}
	f \ast h_t(\exp H)- \what f(i\gamma_p \rho)h_t(\exp H)
	=C_o |W|e^{-|\rho|^2 t-\langle \rho, H\rangle} \sum_{q \in 2Q}e^{-\langle q, H\rangle} E_q(t, H),
\end{equation}
where
 \begin{eqnarray}
\label{bs6}
E_q(t,H)&=&
\int_{\mathfrak{a}} e^{-|\lambda|^2}   \left(\prod_{\alpha'' \in \Sigma_r^+\setminus\Sigma'}(-i\sqrt t\partial_{\alpha''}) \right)
 \frac{\Gamma_q\left(\frac{\lambda}{\sqrt t}+i\frac{H}{2t}\right)}{\hb\left(-\frac{\lambda}{\sqrt t}-i\frac{H}{2t}\right) } U\left(\frac{\lambda}{\sqrt t}+i\frac{H}{2t}\right)\,d\lambda\no\\&&\times\:\: 2^{-|\Sigma_r^+|} t^{-\frac{\nu}{2}}e^{-\frac{|H|^2}{4t}}\sum_{\Sigma'\in \mathcal{P}(\Sigma_r^+)} \left\lbrace \prod_{\alpha' \in \Sigma'} \langle \alpha', H \rangle  \right\rbrace
 \end{eqnarray}
For $j \in \N$, let $\nabla_\lambda^j:=(\frac {\partial^j}{\partial \lambda_1^j}, \frac {\partial^j}{\partial \lambda_2^j}, \cdots, \frac {\partial^j}{\partial \lambda_{\ell}^j})$. Since $f \in C_c^\infty(G//K)$, the Paley--Weiner theorem \cite[p.450]{Helga-3} implies that $\what f(\lambda)$ is a $W$-invariant holomorphic function on $\liea_{\C}$ and there exists a positive constant $C'$ such that for every $j \in \N$ and for every $k \in \N$
\begin{equation}
\label{pw1}
|\nabla_\lambda^j \what f(\lambda)| \le C_{k,j}(1+|\lambda|)^{-k}e^{C'| \Im \lambda |},\:\:\:\:\:\lambda \in \liea_{\C},
\end{equation}
where $C_{k,j}$ is a positive constant depending only on $k$ and $j$. For any fixed $H\in \Omega_t^p$, using \eqref{pw1} together with the mean value theorem in \eqref{ulam}, we get for all large $t$ and $\lambda \in \liea$
\begin{eqnarray}
\label{bs3}
\no\left|U\left(\frac{\lambda}{\sqrt t}+i\frac{H}{2t}\right)\right| &\lesssim& \left |\frac{\lambda}{\sqrt t}+i\frac{H}{2t}- i\gamma_p \rho\right| e^{C'\frac{|H|}{2t}} \\ \no &\le&
\left(\frac{|\lambda|}{\sqrt t}+ \frac{|H-2t \gamma_p \rho|}{2t}\right ) e^{C'\frac{|H|}{2t}}\\ \no &\le&
\left(\frac{|\lambda|}{\sqrt t}+ \frac{r(t)}{2t}\right ) e^{C'\frac{|H|}{2t}}\\
&\lesssim & \left({|\lambda|}+ 1\right ) \frac{r(t)}{t},
\end{eqnarray}
as $H \in \Omega_t^p$ and $1<r(t)/\sqrt t$. Using the definition of $\Omega_t^p$ we observe that
\begin{equation}\label{hbytest}
\frac{|H|}{t}\leq\frac{|H-2t\gamma_p\rho|+2t\gamma_p|\rho|}{t}\leq \frac{r(t)}{t}+2\gamma_p|\rho|\leq 3\gamma_p|\rho|,\:\:\:\:\:\text{for all $t>t_o$}.
\end{equation}
Hence, for all $\alpha \in \Sigma_r^+$ and $t\in (t_o,\infty)$
\begin{equation}
\label{bs1}
1+\left |\left\langle \alpha, -\frac{\lambda}{\sqrt t}-i\frac{H}{2t} \right\rangle  \right| \lesssim \left(1+\frac{|\langle \alpha, \lambda\rangle|}{\sqrt t}\right)\left( 1+ \frac {\langle \alpha, H \rangle}{2t}\right) \lesssim (1+|\lambda|).
\end{equation}
 It now follows from \eqref{bs1} and \eqref{hb1} that
\begin{eqnarray}
\label{bs2}
\left|\hb\left(-\frac{\lambda}{\sqrt t}-i\frac{H}{2t}\right)\right|^{-1} &\lesssim&
 \prod_{\alpha \in \Sigma_r^+}\left(1+\left |\left\langle \alpha, -\frac{\lambda}{\sqrt t}-i\frac{H}{2t} \right\rangle  \right| \right)^{(\frac{m_\alpha+m_{2\alpha}}{2}-1)_+} \no\\&\lesssim& (1+|\lambda|)^m,
\end{eqnarray}
where $m= \displaystyle \sum_{\alpha \in \Sigma_r^+}{(\frac{m_\alpha+m_{2\alpha}}{2}-1)_+}.$

Using \eqref{bs2} in \eqref{hb2}, we get that for all large $t$
\begin{eqnarray}
\label{bs4}
\left|(\sqrt t\nabla_\lambda)^j \left(\hb\left(-\frac{\lambda}{\sqrt t}-i\frac{H}{2t}\right)^{-1}\right) \right|\lesssim
\left|\hb\left(-\frac{\lambda}{\sqrt t}-i\frac{H}{2t}\right)\right|^{-1} \lesssim (1+|\lambda|)^m.
\end{eqnarray}
In view of \eqref{cfs1} and \eqref{cfs2} we also have
\begin{eqnarray}
\label{bs15}
\left|(\sqrt t\nabla_\lambda)^j \Gamma_q\left(\frac{\lambda}{\sqrt t}+i\frac{H}{2t}\right) \right| \lesssim (1+|q|)^d.
\end{eqnarray}
After applying the Leibniz rule in \eqref{bs6}, we notice that $E_q(t,H)$ can be rewritten in the following form
\begin{equation}\label{eqth}
E_q(t,H)=2^{-|\Sigma_r^+|} t^{-\frac{\nu}{2}}e^{-\frac{|H|^2}{4t}}\int_{\m a}e^{-|\la |^2}(T_1(\la,H)+T_2(\la,H))\:d\la,
\end{equation}
where $T_1(\la,H)$ is the sum consisting of terms that do not contain any derivative of $U$, while $T_2(\la,H)$ is the sum of terms that contains at least one derivative of $U$. Now, from \eqref{bs3}, \eqref{bs4} and \eqref{bs15} it follows that
\begin{eqnarray}
\label{bs17}
\no|T_1(\lambda,H)| &\lesssim & (1+|\la|)^{m+1} \left\lbrace \prod_{\alpha' \in \Sigma'} \langle \alpha', H \rangle  \right\rbrace (1+ |q|)^d \frac{r(t)} t  \\
&\lesssim & (1+|\la|)^{m+1} \left( \prod_{\Sigma_r^+}(1+|H|) \right) (1+ |q|)^d \frac{r(t)} t
\no \\ &\lesssim & t^{|\Sigma_r^+|}(1+|\la|)^{m+1} (1+ |q|)^d \frac{r(t)} t,
\end{eqnarray}
for all large $t$, where we have used \eqref{hbytest} in the last inequality. Similarly, from \eqref{pw1}, \eqref{bs4} and \eqref{bs15} we get
\begin{eqnarray}
\label{bs18}
|T_2(\la, H) | &\lesssim& (1+|\la|)^m \left\lbrace \prod_{\alpha' \in \Sigma'} \langle \alpha', H \rangle  \right\rbrace (1+ |q|)^d \no \\
&\lesssim &(1+|\la|)^m (1+|H|)^{|\Sigma_r^+|-1}(1+ |q|)^d,\hspace{1.5cm}(\text{as $|\Sigma'| <  |\Sigma_r^+|$})\no\\
\no\\
&\lesssim & t^{|\Sigma_r^+|-1}(1+|\la|)^m (1+ |q|)^d\hspace{3cm} (\text{as $1+|H| \lesssim t$})\no\\
 \no \\ &\lesssim & t^{|\Sigma_r^+|}(1+|\la|)^{m+1} (1+ |q|)^d \frac{r(t)} t,
\end{eqnarray}
as $r(t)\in (1,\infty)$ for all large $t$.
Hence, combining \eqref{eqth}, \eqref{bs17} and \eqref{bs18} we get
\begin{equation}
\label{bs19}
|E_q(t, H) |\lesssim t^{-\frac{\ell}{2}}e^{-\frac{|H|^2}{4t}} (1+ |q|)^d \frac{r(t)} t,\hspace{1.5cm}\text{as \:\:$\nu= \ell + 2|\Sigma_r^+| $}.
\end{equation}
Thus, for all sufficiently large $t$ and $H \in \Omega_t^p$, from \eqref{lts-1} and \eqref{bs19} we get
 \begin{eqnarray}\label{finalest}
 &&|f \ast h_t(\exp H)- \what f(i\gamma_p \rho)h_t (\exp H)|\no\\
 &\lesssim&\frac{r(t)} t t^{-\frac{\ell}{2}} e^{-|\rho|^2t-\langle \rho, H\rangle-\frac{|H|^2}{4t}} \overbrace{\sum_{q \in 2Q}(1+ |q|)^d e^{-\langle q, H\rangle}}^{\lesssim 1}\lesssim   \frac{r(t)} t t^{-\frac{\ell}{2}} g_t(H),
 \end{eqnarray}
where we recall the definition of $g_t(H)$ from \eqref{gtexph}. Integrating according to the  Cartan decomposition and using \eqref{finalest}, \eqref{bs322}  successively, we get
\begin{eqnarray}
&&\|h_t\|_p^{-p}\int_{K(\exp \Omega_t^p) K}|f \ast h_t(x)- \what f(i\gamma_p \rho)h_t (x)|^p \, dx \no \\ \no
&\lesssim& {\left( \frac{r(t)} t\right)}^p t^{-\frac{\ell p}{2}}\int_{\Omega_t^p} \|h_t\|_p^{-p}g_t(H)^pe^{2\langle \rho, H\rangle} \, dH\no \\
&\lesssim& {\left( \frac{r(t)} t\right)}^p t^{-\frac{\ell p}{2}}\int_{\Omega_t^p}t^{\frac{\ell p}{2p'}} e^{-\frac p{4t}|H-2t \gamma_p \rho|^2}\, dH\no \\ \no
&\lesssim& {\left( \frac{r(t)} t\right)}^p t^{-\frac{\ell }{2}}\int_{B(2t\gamma_p \rho, r(t))}e^{-\frac p{4t}|H-2t \gamma_p \rho|^2}\, dH \\ \no
&\lesssim& {\left( \frac{r(t)} t\right)}^p t^{-\frac{\ell }{2}}\int_0^{r(t)}e^{-\frac {ps^2}{4t}} s^{l-1}\, ds\no \\
&\lesssim& {\left( \frac{r(t)} t\right)}^p \int_0^{\frac{r(t)}{\sqrt t}}e^{-\frac {ps^2}{4}} s^{l-1}\, ds\lesssim{\left( \frac{r(t)} t\right)}^p \int_0^{\infty}e^{-\frac {ps^2}{4}} s^{l-1}\, ds\lesssim{\left( \frac{r(t)} t\right)}^p. \no
\end{eqnarray}
The result now follows from the fact that $r(t)/t$ goes to zero as $t$ tends to infinity.
\end{proof}

\begin{proposition}
\label{bs47}
If $p\in (1,2)$ and $f \in C_c^\infty(G//K)$ then
\begin{equation}
\lim_{t \to \infty}\|h_t\|_p ^{-1}\|f \ast h_t -\what f(i \gamma_p \rho)h_t\|_{L^p(G\setminus K(\exp \Omega_t^p) K)}=0.
\end{equation}
\end{proposition}
\begin{proof}
We assume that the support of $f$ is contained in $K(\exp B(0,\delta))K$ for some $\delta>0$. By using the Minkowski integral inequality, we get
\begin{eqnarray}
\label{cl1}
&&\|f \ast h_t\|_{L^p(G\setminus K(\exp \Omega_t^p) K)}\no\\
&=& \left(\int\limits_{G\setminus K(\exp \Omega_t^p) K}\left|\int \limits_{K(\exp B(0,\delta))K}f(y) h_t(xy^{-1})\,dy\right|^p \,dx\right)^{\frac 1p}\no\\
&\le&  \int \limits_{K(\exp B(0,\delta))K} |f(y)|\left(\int\limits_{G\setminus K(\exp \Omega_t^p) K} h_t(xy^{-1})^p\,dx\right)^{\frac 1p} \,dy.
\end{eqnarray}
For $x\in G\setminus K(\exp \Omega_t^p) K$ and $y\in K(\exp B(0,\delta))K$, it follows that
\begin{equation*}
xy^{-1}\in G\setminus K(\exp B(2t\gamma_p \rho, r(t)-\delta))K.
\end{equation*}
This was proved for $p=1$ in \cite[p.16]{AE} but the same proof works in our case. Using this observation in the inner integral on the right hand side of \eqref{cl1}, we obtain

\begin{eqnarray}
\label{cl4}
&&\|h_t\|_p ^{-1} \|f \ast h_t\|_{L^p(G\setminus K(\exp \Omega_t^p) K)}\no\\
&\le&\|h_t\|_p ^{-1} \int \limits_{K(\exp B(0,\delta))K} |f(y)|\left(\int\limits_{G\setminus K(\exp B(2t\gamma_p \rho, r(t)-\delta)) K} h_t(z)^p\,dz\right)^{\frac 1p} \,dy\no\\
&=&\|f\|_1\|h_t\|_p ^{-1}\|h_t\|_{L^p(G\setminus K(\exp B(2t\gamma_p \rho, r(t)-\delta)) K)}.
\end{eqnarray}
From Proposition \ref{cl5} and \eqref{cl4} it is clear that
\begin{equation}
\label{cl8}
\lim_{t \to \infty}\|h_t\|_p ^{-1} \|f \ast h_t\|_{L^p(G\setminus K(\exp \Omega_t^p) K)}=0.
\end{equation}
Since
 \begin{eqnarray*}
&& \|h_t\|_p ^{-1}\|f \ast h_t -\what f(i \gamma_p \rho)h_t\|_{L^p(G\setminus K(\exp \Omega_t^p) K)} \no \\
 & \leq &\|h_t\|_p ^{-1}\|f \ast h_t\|_{L^p(G\setminus K(\exp \Omega_t^p) K)}+|\what f(i \gamma_p \rho)|\|h_t\|_p ^{-1}\|h_t\|_{L^p(G\setminus K(\exp \Omega_t^p) K)},
\end{eqnarray*}
the result follows from \eqref{cl8} and Proposition \ref{cl5}.
\end{proof}
By combining Proposition \ref{cl10}, Proposition \ref{bs47}, and Proposition \ref{p=2}, we obtain the proof of Theorem \ref{cl16}, (a) for the special case of $f\in C_c^\infty(G//K)$:
\begin{equation}\label{mlemma}
\lim_{t \to \infty}\|h_t\|_p ^{-1}\|f \ast h_t -\what f(i \gamma_p \rho)h_t\|_p=0,\:\:\:p\in(1,2].
\end{equation}

\begin{proof}[\textbf{Completion of the proof of Theorem \ref{cl16}}]
We first prove $(a)$, that is, we assume that $p\in (1,2]$ (the case $p=1$ being proved in \cite{AE}). Let $\epsilon >0$ be given. As $C_c^\infty(G//K)$ is dense in  $\mathcal L^p(G//K)$, we choose $g \in C_c^\infty(G//K)$ such that
\begin{equation}
\label{cl18}
| \what g (i\gamma_p \rho) -\what f(i \gamma_p \rho) |
\le \what{|f-g|}(i\gamma_p \rho) <\epsilon.
\end{equation}
 Applying the Herz criterion \eqref{hqwe1}, we get
\begin{equation}
\label{cl15}
\|(f-g )\ast h_t \|_p \le\what{|f-g|}(i\gamma_p \rho)  \|h_t\|_p < \epsilon \|h_t\|_p .
\end{equation}
As $g \in C_c^\infty(G//K)$, using  \eqref{mlemma}, we obtain a $t_o \in(0,\infty)$ such that
\begin{equation}
\label{cl20}
\|g \ast h_t -\what g(i \gamma_p \rho)h_t\|_{p} < \epsilon\|h_t\|_p,
\end{equation}
for all $t > t_o.$
Using \eqref{cl18}-\eqref{cl20}, it follows that for all $t>t_o$
\begin{eqnarray*}
&&\no\|f \ast h_t -\what f(i \gamma_p \rho)h_t\|_p \\\no
&\le &\|(f-g )\ast h_t \|_p+ \| g \ast h_t -\what g(i \gamma_p \rho)h_t\|_p +| \what g (i\gamma_p \rho) -\what f(i \gamma_p \rho)|\|h_t\|_p
\\ &<& 3\epsilon \|h_t \|_p.
\end{eqnarray*}
Hence,
\begin{equation*}
\lim_{t \to \infty}\|h_t\|_p ^{-1}\|f \ast h_t -\what f(i \gamma_p \rho)h_t\|_{p}=0.
\end{equation*}

To prove $(b)$, we first show that if $f\in \mathcal L^2(G//K)$ then $f\ast h_t\in L^{\infty}(G//K)$. From the pointwise estimate of $h_t$ (see \eqref{hes-1}) it follows that for each $t\geq 1$, there exists a positive constant $C_t$ such that
\begin{equation*}
\frac{h_t(\exp H)}{\varphi_0(\exp H)}\leq C_t,\:\:\:\:\:\:\text{for all $H\in \overline{\m a_+}$}.
\end{equation*}
Hence, for all $x\in G$,
\begin{eqnarray}\label{point}
|f\ast h_t(x)|&\le & \int_G |f(y^{-1}x)|h_t(y)\:dy\nonumber\\
&=& \int_G |f(y^{-1}x)|\varphi_0(y)\frac{h_t(y)}{\varphi_0(y)}\:dy\nonumber\\
&\le & C_t|f|\ast\varphi_0(x)=C_t\what {|f|} (0)\varphi_0(x),
\end{eqnarray}
by the relation \eqref{spproj}. This implies that $f\ast h_t\in L^{\infty}(G//K)$. Since $f\ast h_t\in L^2(G//K)$ (see \eqref{hqwe1}), it follows by convexity that $f\ast h_t\in L^p(G//K)$, for all $p\in (2,\infty]$. The convolution inequality \eqref{cks} and the semigroup property of $h_t$ now implies that for any $p\in (2,\infty]$,
\begin{eqnarray}
\|f\ast h_t-\what f(0)h_t\|_p&=& \|f\ast h_{t/2}\ast h_{t/2}-\what f(0) h_{t/2}\ast h_{t/2}\|_p \no\\
&\le &  \|f\ast h_{t/2}-\what f(0) h_{t/2}\|_2\| h_{t/2}\|_2.
\no
\end{eqnarray}
Hence,
\begin{equation}\label{comparison}
\frac{\|f\ast h_t-\what f(0)h_t\|_p}{\|h_t\|_p}\leq \frac{\|f\ast h_{t/2}-\what f(0)h_{t/2}\|_2}{\|h_{t/2}\|_2}\times \frac{\|h_{t/2}\|_2^2}{\|h_t\|_p}.
\end{equation}
The estimates in \eqref{ehe-1} show that for all large $t$,
\begin{equation*}
\frac{\|h_{t/2}\|_2^2}{\|h_t\|_p}\asymp \left( \frac{t}{2}\right)^{-\frac{2\nu}{4}}e^{-2|\rho|^2t/2}t^{\frac{\nu}{2}}e^{|\rho|^2 t}\asymp 1.
\end{equation*}
Using this observation in \eqref{comparison} we obtain
\begin{equation}
\label{ra1}
\frac{\|f\ast h_t-\what f(0)h_t\|_p}{\|h_t\|_p}\le C \frac{\|f\ast h_{t/2}-\what f(0)h_{t/2}\|_2}{\|h_{t/2}\|_2}.
\end{equation}
So, as $\gamma_2=0$, it follows from part $(a)$ that the ratio in the left hand side of \eqref{ra1} goes to zero as $t$ goes to infinity. This completes the proof of part (b).
\end{proof}

\begin{rem}\label{withoutkz}
\textup{Using the Fourier inversion formula \eqref{inv-1}, one can prove the second part of Theorem \ref{cl16}, namely the case $p\in(2,\infty]$, without using the semigroup property of $h_t$ and the Kunze--Stein phenomenon. Indeed, as $f\in\mathcal{L}^2(G//K)$, $\widehat{f}$ is a bounded continuous function on $\mathfrak{a}$. Thus, using \eqref{inv-1} we write
\begin{equation*}
	f\ast h_t(x)-\what{f}(0)h_t(x)=C_o\int_{\mathfrak{a}}\left(\what f(\la)-\what f(0)\right)e^{-(|\la|^2+|\rho|^2)t}\varphi_{\la}(x)|\hc(\la)|^{-2}\:d\la,\:\:\:x\in X.
	\end{equation*}
Applying the Minkowski's inequality and the estmate $\|\varphi_{\la}\|_p\leq\|\varphi_{0}\|_p$, we obtain
\begin{equation*}
\|f\ast h_t-\what f(0)h_t\|_p\leq C_o\|\varphi_{0}\|_p\int_{\mathfrak{a}}|\what f(\la)-\what f(0)|e^{-(|\la|^2+|\rho|^2)t}|\hc(\la)|^{-2}\:d\la.
\end{equation*}
Using the change of variable $\displaystyle{\la\mapsto\frac{\la}{\sqrt{t}}}$, and proceeding as in the proof of Proposition \ref{p=2},  we get
\begin{equation*}
\|h_t\|_p^{-1}\|f\ast h_t-\what f(0)h_t\|_p\lesssim\int_{\m a} |\what {f}(\frac{\la}{\sqrt{t}})-\what f(0)| e^{-|\lambda|^2}|P(\la)|\:d\lambda,
	\end{equation*}
where $P(\la)$ is a polynomial in $\la$. We complete the proof by applying the dominated convergence theorem on the right hand side of the inequality above.}
\end{rem}

\section{Results for fractional heat  kernel}
\label{sef1}In this section we give the proof of Theorem \ref{fcl16}. To begin with let us recall the estimates \eqref{fehe-1} of  $\|h_t^\alpha\|_p$.
\begin{proof}[\textbf{Proof of Theorem \ref{fcl16}}]
We first prove part $(a)$, that is, we assume that $p\in [1,2]$. Without loss of generality we may assume that $f \in C_c^\infty(G//K)$, for, using standard density arguments as in the proof of Theorem \ref{cl16}, we can always extend the result to $\mathcal L^p(G//K)$. From  \eqref{ehe-1} and \eqref{fehe-1} it is clear that for all sufficiently large $t$,
\begin{equation}\label{normfracest}
\|h_t\|_p \asymp t^{-a}e^{-tb},\:\:\: \:\: \|h_t^\alpha\|_p \asymp t^{-a} e^{-tb^\alpha},
\end{equation}
for some non-negative constants $a$, $b$ depending on $p$. Let $\epsilon >0$ be given. Utilizing Theorem \ref{cl16}, we get $M_1 >0$ such that
\begin{equation}\label{usingheat}
\|h_u\|_p^{-1}\| f \ast h_u-\what f(i \gamma_p \rho)h_u\|_p  < \epsilon,\hspace{0.7cm}\text{for all $u > M_1$.}
\end{equation}
From \eqref{fhss1} it follows that
\begin{equation*}
f \ast h_t^\alpha(x) -\what f(i \gamma_p \rho) h_t^\alpha(x)
= \int_{0}^\infty \left( f \ast h_u(x) -\what f(i \gamma_p \rho)h_u (x) \right) \, \eta_t^{\alpha}(u)\, du.
\end{equation*}
Thus, using the Minkowski integral  inequality, we get
\begin{equation}
\label{fhe1}
\|f \ast h_t^\alpha -\what f(i \gamma_p \rho) h_t^\alpha\|_p
\le\int_{0}^\infty \| f \ast h_u-\what f(i \gamma_p \rho)h_u\|_p  \, \eta_t^{\alpha}(u)\, du.
\end{equation}
After recalling $c_{\alpha}$ from \eqref{fehe-2}, we choose a positive number $ L$  such that
$$\delta :=c_\alpha L^{-\frac{\alpha}{(1-\alpha)}}-b^\alpha>0.$$
From \eqref{fhe1} we write
\begin{equation}
\label{fhe11}
\|h_t^\alpha \|_p^{-1}\|f \ast h_t^\alpha -\what f(i \gamma_p \rho) h_t^\alpha\|_p
\le A_t +B_t,
\end{equation}
where
\begin{eqnarray}
A_t
=\|h_t^\alpha \|_p^{-1} \int_{0}^{Lt} \| f \ast h_u-\what f(i \gamma_p \rho)h_u\|_p  \, \eta_t^{\alpha}(u)\, du,\label{fhe4}\\
 B_t
=\|h_t^\alpha \|_p^{-1} \int_{Lt}^\infty \| f \ast h_u-\what f(i \gamma_p \rho)h_u\|_p  \, \eta_t^{\alpha}(u)\, du\no.
\end{eqnarray}
Using \eqref{usingheat}, we have for all $t > \max\{1,M_1/L\}$,
\begin{eqnarray}
 B_t&<& \epsilon \|h_t^\alpha \|_p^{-1} \int_{Lt}^\infty \|h_u\|_p  \, \eta_t^{\alpha}(u)\, du \no\\
 &\lesssim & \epsilon  e^{tb^\alpha}  \int_{Lt}^\infty  e^{-ub}\, \eta_t^{\alpha}(u)\, du\no\\
 &&\hspace{2.5cm}\text{(using the estimates \eqref{normfracest}, and  $u^{-a}\lesssim t^{-a}$, as $u\geq Lt$)}\no
\\ &\lesssim & \epsilon  e^{tb^\alpha}  \int_{0}^\infty  e^{-ub}\, \eta_t^{\alpha}(u)\, du
\lesssim \epsilon, \label{fhe3}
\end{eqnarray}
where we have used \eqref{fhss2} in the last step. Since $$\| f \ast h_u-\what f(i \gamma_p \rho)h_u\|_p  \lesssim \|h_u\|_p \le u^{-a},$$ using \eqref{fehe-3} in \eqref{fhe4} we get
\begin{eqnarray}
 A_t
\lesssim \|h_t^\alpha \|_p^{-1} \int_{0}^{Lt}  u^{-a}\,  t^{a_1}(1+ u^{a_2})e^{-c_\alpha t^{\frac{1}{1-\alpha}}u^{-\frac{\alpha}{1-\alpha}} }\, du \label{fhe44}.
 \end{eqnarray}
Using the substitution
\begin{equation*}
	v=t^{\frac{\alpha}{1-\alpha}}u^{-\frac{\alpha}{1-\alpha}}
\end{equation*}
 in \eqref{fhe44}, we obtain constants $a_i\in\R$, $i=3,4,5,6$, such that
\begin{eqnarray*}
 A_t
&\lesssim&\|h_t^\alpha \|_p^{-1} t^{a_3} \int\displaylimits_{L^{-\frac{\alpha}{1-\alpha}}}^{\infty}  v^{a_4}\,  (1+ t^{a_5}v^{a_6})e^{-c_\alpha t v}\:dv\no \\
&\lesssim & t^{a_7} \int\displaylimits_{L^{-\frac{\alpha}{1-\alpha}}}^{\infty}  v^{a_4}\,  (1+ t^{a_5}v^{a_6})e^{-t(c_\alpha v-b^\alpha)}\, dv,
\end{eqnarray*}
where $a_7=a+a_3$, and we have used estimate of $\|h_t^\alpha \|_p$ (see \eqref{normfracest}). Making the substitution $u=c_\alpha v-b^\alpha$, in the right hand side of the inequality above and noting that $u+b^\alpha\asymp u$, we get
\begin{eqnarray}
A_t
&\lesssim&  t^{a_7} \int\displaylimits_{\delta}^{\infty}  u^{a_4}\,  (1+ t^{a_5}u^{a_6})e^{-tu }\, du \hspace{1.5cm}\text{(as $\delta=c_\alpha L^{-\frac{\alpha}{(1-\alpha)}}-b^\alpha$)}\no \\
&\lesssim&  t^{a_8} \int\displaylimits_{\delta t}^{\infty}  u^{a_4}\,  (1+ t^{a_9}u^{a_6})e^{-u }\, du,
\label{fhe45}
\end{eqnarray}
for some $a_8\in\R$, $a_9 \in \R$. From \eqref{fhe11}, \eqref{fhe3} and \eqref{fhe45} it follows that
\begin{equation}
\|h_t^\alpha \|_p^{-1}\|f \ast h_t^\alpha -\what f(i \gamma_p \rho) h_t^\alpha\|_p
\lesssim  \epsilon,
\end{equation}
for all sufficiently large $t$. This completes the proof of Theorem \ref{cl16}, $(a)$.

The proof of Theorem \ref{fcl16}, $(b)$ is analogous to that of Theorem \ref{cl16}, $(b)$ due to the convolution inequality \eqref{cks}
and \eqref{ep10}.
\end{proof}

\begin{rem}
\label{irem-1} \textup{Let $p \in [1,2]$ and $\alpha \in (0,1]$ be fixed and let $f \in \mathcal L^p(G//K)$ be non-negative. We note that
\begin{equation*}
\what f(i \gamma_p \rho)=\frac{ \|\what f(i \gamma_p \rho) h_t^\alpha\|_p}{\|h_t^\alpha\|_p}\le \frac{ \|f \ast h_t^\alpha -\what f(i \gamma_p \rho) h_t^\alpha\|_p}{\|h_t^\alpha\|_p} +\frac{ \|f \ast h_t^\alpha\|_p}{\|h_t^\alpha\|_p}.
\end{equation*}
Hence, by using \eqref{hqwe1} it follows from above that
\begin{equation}\label{hpjk1}
\what f(i \gamma_p \rho)-\frac{ \|f \ast h_t^\alpha -\what f(i \gamma_p \rho) h_t^\alpha\|_p}{\|h_t^\alpha\|_p}\le\frac{\|f \ast h_t^\alpha \|_p}{\|h_t^\alpha\|_p}\le \what f(i\gamma_p \rho),
\end{equation}
Letting $t$ tending to infinity in \eqref{hpjk1} we get
\begin{equation}
\label{hpjk2}
\what f(i \gamma_p \rho)=|||T_f|||_p=\lim_{t\to \infty} \frac{ \|f \ast h_t^\alpha\|_p}{\|h_t^\alpha\|_p}.
\end{equation}
The relation \eqref{hpjk2} precisely says that in order to get the operator norm of $T_f$ it is enough to look
at the asymptotic limit  of action of $T_f$ on the family of unit vectors $\{\|h_t^\alpha\|_p^{-1}h_t^\alpha \mid t>0\}$ instead of
whole unit sphere of $L^p(G//K)$. In other words, the family $\{\|h_t^\alpha\|_p^{-1}h_t^\alpha \mid t>0\}$ is an extremizer for the operator $T_f$.}
\end{rem}

\subsection{Counter example for the  case when $f$ is not $K$-biinvariant }
It was shown in  \cite {Vaz-2, AE} that Theorem \ref{cl16} is not true in general for functions $f\in L^1(X)$ which are not left $K$-invariant. Theorem \ref{fcl16} for the case $\alpha=1/2$,  $p=1$ was proved in \cite{Eff1, Eff2}, where the author has constructed a similar counterexample. Construction of such a counterexample (as given in \cite {Vaz-2, AE, Eff1, Eff2}) depends crucially on the $L^1$ concentration of the corresponding kernel (namely, heat kernel and the Poisson kernel) in the critical region along with their behaviour under left translations by elements of $G$, in the critical region. In this subsection we show by means of a counterexample that Theorem \ref{cl16} and Theorem \ref{fcl16} are not true in general for functions $f\in L^1(X)$ which are not left $K$-invariant. Our method of constructing such counterexamples for $p=1,2$ and $\alpha \in (0,1]$ is different from the above mentioned papers. Though we give counterexamples only for the cases $p=1$ and $p=2$, we believe that it should be possible to get counterexamples for other values of $p$ as well.

Let $0 \ne f \in L^1(G//K)$ be  a non-negative function. As $f$ is non-negative, it is clear that
$\what f(i \gamma_p \rho) \ne 0$ for all $p \in [1,2]$. Let $y \in G\setminus K$ and we define the right $K$-invariant function $g:G \to \C$ given by
\begin{equation}
g(x)= f(y^{-1}x),\:\:\:\:\:x\in G.
\end{equation}
We shall deal with two cases separately.

{\bf Case I ($L^1$-case).} Suppose that \begin{equation}\label{cec1}
 \lim_{t\to\infty}{\|g \ast h_t^\alpha-\widehat{g}(i\rho)h_t^\alpha\|_1}=0.
\end{equation}
We now use \eqref{helgs1} and \eqref{helgs2} to obtain the following inequality
\begin{eqnarray} \no
\int_{K/M}\left|e^{\langle 2\rho, A(k^{-1}y) \rangle}-1\right|\what f(i\rho)dk_M&=&
\int_{K/M}\left|\wtilde {g \ast h_t^\alpha}(i\rho, kM)-\widehat{g}(i\rho)\wtilde {h_t^\alpha}(i\rho, kM)\right|dk_M \\& \lesssim &\|g \ast h_t^\alpha-\widehat{g}(i\rho)h_t^\alpha\|_1.
\label{cec2}
\end{eqnarray}
From \eqref{cec1}, \eqref{cec2} and continuity of the map $k\mapsto A(k^{-1}y)$ it now follows that for all $k \in K$
\begin{equation}\label{cfr12}
e^{\langle 2\rho, A(k^{-1}y) \rangle}= 1.
\end{equation}
As $y \notin K$, we must have $ y= k_3\exp (y^+)k_4$ with $y^+\in \m a_+$. Therefore, $\langle \rho, y^+ \rangle$ is in $\R^+$. Now, taking
$k=k_3$ in \eqref{cfr12} we get
\begin{equation}
\label{cfr13}
1=e^{\langle 2\rho, A(k_3^{-1}y)\rangle}=e^{\langle 2\rho, y^+\rangle}>1,
\end{equation}
which is a contradiction. Hence, \eqref{cec1} must be false.

 {\bf Case-II ($L^2$-case).} We now suppose that
 \begin{equation}
 \label{as1}
  \lim_{t\to\infty}\frac{\|g \ast h_t^\alpha-\widehat{g}(0)h_t^\alpha\|_2}{\|h_t^\alpha\|_2}=0.
 \end{equation}

From Plancherel theorem \cite[Ch. III]{Helga-2} and the estimate \eqref{fehe-1} of $\|h_t^{\alpha}\|_2$,  it follows that
\begin{eqnarray}\label{fthq1}
&&\|h_t^\alpha\|_2^{-2}\|g \ast h_t^\alpha-\widehat{g}(0)h_t^\alpha\|^2_2 \no \\
&\asymp& t^{\frac \nu 2} \int_{K/M}  \int_{\m a} \left|e^{\left\langle -i\lambda+\rho, A(k^{-1}y)\right\rangle} \what f(\la)-\what f(0)\varphi_0(y)\right|^2 e^{-2t[(|\lambda|^2+|\rho|^2)^{\alpha}-|\rho|^{2\alpha}]}  |\hc(\lambda)|^{-2}\, d\lambda \,dk_M\no \\
 &\ge& t^{\frac \nu 2}  \int\displaylimits_{ K/M}  \int\displaylimits_{|\lambda| \le 1 } \left|e^{\left\langle -i\lambda+\rho, A(k^{-1}y)\right\rangle}\what f(\la) -\varphi_0(y)\what f(0)\right|^2 e^{-2t\left[(|\lambda|^2+|\rho|^2)^{\alpha}-|\rho|^{2\alpha}\right]} \no\\&&\hspace{2cm}\times |\hc(\lambda)|^{-2}\, d\lambda \,dk_M.
\label{fclhm2}
\end{eqnarray}
For $|\lambda| \le 1$, we also have
\begin{equation}
\label{f-1}
(|\lambda|^2+|\rho|^2)^{\alpha}-|\rho|^{2\alpha} \asymp |\lambda|^2.
\end{equation}
Hence, it follows that there exists positive constants $A$ and $B$ such that for all  $|\lambda| \le 1$
\begin{equation}
\label{f-2}
e^{-A t |\lambda|^2} \le e^{-2t\left[(|\lambda|^2+|\rho|^2)^{\alpha}-|\rho|^{2\alpha}\right]}  \le e^{-Bt |\lambda|^2}.
\end{equation}
Employing \eqref{f-2} in \eqref{fclhm2}, we get
\begin{eqnarray}\label{mhg23}
&&\|h_t^\alpha\|^{-2}_2\|g \ast h_t^\alpha-\widehat{g}(0)h_t^\alpha\|^2_2 \no \\
 &\ge& t^{\frac \nu 2}  \int_{ K/M}  \int_{|\lambda| \le 1 } \left|e^{\left\langle -i\lambda+\rho, A(k^{-1}y)\right\rangle} \what f(\la)-\what f(0)\varphi_0(y)\right|^2 e^{-At|\lambda|^2}  |\hc(\lambda)|^{-2}\, d\lambda \,dk_M  \no \\
  &\gtrsim& t^{\frac \nu 2}  \int_{ K/M}  \int_{|\lambda| \le 1 } \left|e^{\left\langle -i\lambda+\rho, A(k^{-1}y)\right\rangle} \what f(\la)-\what f(0)\varphi_0(y)\right|^2| e^{-At|\lambda|^2}\no\\&&\hspace{2cm}\times\left[ \prod_{\alpha \in \Sigma_r^+}| \langle \alpha, \lambda \rangle|^{2}(1+| \langle \alpha, \lambda \rangle|)^{{m_\alpha+m_{2\alpha}}-2}\right]\, d\lambda \,dk_M.
\end{eqnarray}
Using the change of variable $\lambda \to \frac{\lambda}{\sqrt t}$ on the right hand side of \eqref{mhg23} and then taking limit as $t \to \infty$ on both sides of \eqref{mhg23}, we get using \eqref{as1}
\begin{eqnarray}
 \int_{ K/M}  \int_{|\la |\leq 1} \left|e^{\left\langle \rho, A(k^{-1}y)\right\rangle} -\varphi_0(y)\right|^2|\what {f}(0)|^2 e^{-A|\lambda|^2}   \left[ \prod_{\alpha \in \Sigma_r^+}| \langle \alpha, \lambda \rangle|^{2}\right]\, d\lambda \,dk_M =0.
 \label{mnw1}
\end{eqnarray}
Since $\what f(0)$ is nonzero, from \eqref{mnw1} we get
\begin{equation}
\label{f-6}
\int_{ K/M}  \left|e^{\left\langle \rho, A(k^{-1}y)\right\rangle} -\varphi_0(y)\right|^2\, dk_M=0.
\end{equation}
It then follows, as before, that for all $k\in K$
\begin{equation*}
e^{\left\langle \rho, A(k^{-1}y)\right\rangle} = \varphi_0(y).
\end{equation*}
Now, choosing $k$ suitably (depending on $y$) and substituting in the relation above, as in the $L^1$-case, we obtain
\begin{equation*}
	1\geq \varphi_0(\exp y^+)=e^{\langle\rho,y^+\rangle}>1,
\end{equation*}
which is  a contradiction. Hence, the supposition \eqref{as1} is false.

\section{Rank one results: Optimality and counterexamples}
\label{opc}
In this section, we shall first discuss about the optimality of the rate of convergence in Theorem \ref{cl16} for rank one symmetric spaces of noncompact type. For higher rank symmetric spaces or for $p\in[1,2)$, we still don't know whether  such optimality in Theorem \ref{cl16}  holds or not.  However, we show that if $p\in[2,\infty]$ and $f \in L^q(G//K)$ for some $q\in[1,2)$, then it is possible to get an improved version of  Theorem \ref{cl16} (see Proposition \ref{sness-1}(a)). We also show that the result is sharp for $p=2$ (see Proposition \ref{sness-1}(b)) provided we assume  $f \in L^q(G//K)$ for some $q\in[1,2)$. Since Theorem \ref{cl16} is true for a much larger class of functions, namely $f\in \mathcal L^p(G//K)$, at this moment we do not know anything about its sharpness in this larger class. From now onwards, we assume that $X$ is a rank one symmetric space of noncompact type, that is, $G$ is a connected noncompact semisimple Lie group of real rank one. Thus the Lie group $A$ is isomorphic to $\R$ and by fixing a suitable isomorphism, we write $A =\{ a_t \mid t \in \R \}$. We will need the following facts for proving Proposition \ref{sness-1}.

For a suitable $K$-biinvariant  function $f$ on $G$,  the Abel transform of $f$ is defined by
$$\mathscr Af(t)=e^{\rho t }\int_{N}f(a_tn)\, dn,\,\,\,\,\,\,t \in \R.$$
For a nice function $g$ on $\R$, let
$$\mathscr{F}g(\lambda)=\int_{\R }g(s) e^{-i\lambda s}\, ds, \,\,\,\,\,\,\lambda \in \R.$$
be the classical Fourier transform of $g$,

Let $f\in L^p(G//K)$ for $p\in[1,2)$. Then
 $\what f(\lambda)$ exists for all $\lambda  \in \C$ with $|\Im  \lambda| < \gamma_p \rho$, and
 defines a holomorphic function there. For such a function $f$, its Abel transform $\mathscr Af$ exists as an even function on $\R$ with
\begin{equation}
\int_{\R}|\mathscr Af(t)|e^{\beta|t|}\, dt \lesssim \|f\|_p,
\label{oat1}
\end{equation}
for any $0\le \beta <\gamma_p \rho$, and
\begin{equation}
\mathscr{F} (\mathscr Af)(\lambda)= \what f (\lambda),\,\,\,\,\,\,\text{whenever}\,\,\,|\Im \lambda| < \gamma_p \rho.
\label{oat2}
\end{equation}
We observe that the $\hc$-function has the following behaviour in rank one symmetric space of noncompact type(see \eqref{cfest}):
\begin{equation}
\label{cfest10}
|\hc(\lambda)| \asymp |\lambda|^{-1}(1+| \lambda|)^{-\frac{n-3}2}, \,\,\,\,\,\,\text{whenever}\,\,\,\Im \lambda \le 0.
\end{equation}
\begin{proposition}
\label{sness-1}
 Let $\theta:(0,\infty) \to [0,\infty) $.
\begin{itemize}
\item [(a)] Suppose $\lim_{t \to \infty}\theta(t)/t=0$, $q \in [1,2)$, $p\in[2,\infty]$.
Then for any $f \in L^q(G//K)$,
\begin{equation*}
\lim_{t \to \infty}{\theta(t)}\|h_t\|_p^{-1}\|f \ast h_t-\what f(0)h_t\|_p=0.
\end{equation*}
\item [(b)] Suppose $\lim_{t \to \infty}\theta(t)/t>0$. Then there exists  $f \in L^q(G//K)$
 with $q\in[1,\infty]$ such that
\begin{equation*}
\limsup_{t \to \infty}{\theta(t)}\|h_t\|_2^{-1}\|f \ast h_t-\what f(0)h_t\|_2 > 0.
\end{equation*}
\end{itemize}
\end{proposition}
\begin{proof}
\begin{itemize}

\item[(a)]We first consider the case $p=2$. From  Plancherel formula, \eqref{ehe-1} and \eqref{cfest10} it follows that
\begin{eqnarray}
&\,&\|h_t\|_2^{-2}\|f \ast h_t- \what f(0) h_t\|_2^2 \no \\
&\asymp& t^{\frac 3 2}  e^{2t\rho^2}  \int_0^\infty |\what {f}(\lambda)-\what f(0)|^2 e^{-2t(\lambda ^2+\rho ^2)}  |\hc(\lambda)|^{-2}\, d\lambda \no \\
&\asymp& t^{\frac 3 2} \int_{0}^\infty |\what {f}(\lambda)-\what f (0)|^2 e^{-2t \lambda ^2}
\lambda^{2}(1+ \lambda )^{{n-3}}  \, d\lambda \label{hsre1}\\
&\asymp & \int_{0}^\infty \left|\what {f}\left(\frac{\lambda}{\sqrt t}\right)- \what f(0)\right|^2 e^{-2 \lambda^2}
\lambda^{2}\left(1+\frac{\lambda}{\sqrt t}\right)^{{n-3}}  \, d\lambda.
\label{flhm200}
\end{eqnarray}
Since $\what f$ is an even real analytic function, by expanding $\what f$  around zero, we get
\begin{equation}
\what {f}\left(\frac{\lambda}{\sqrt t}\right)- \what f(0)= \frac{ \lambda^2}{t}\,({\what f})^{''}(s_{\lambda,t} ),
\label{soe12}
\end{equation}
for some $s_{\lambda, t} \in \R$ depending on $\lambda, t$.
We define a function $h: \R\to \C$ by
\[h(s)=s^2\mathscr A f(s), \,\,\,\, s \in \R.\]
Owing to \eqref{oat1} and \eqref{oat2} we obtain
\begin{equation*}
|(\what f )^{''}( u)|=|(\mathscr{F}(\mathscr A f))^{''}(u)|
= |\mathscr{F}(h)(u)|
\leq\|\mathscr{F}(h)\|_{L^{\infty}(\R)}.
\end{equation*}
Hence
\begin{equation}\label{seq1}
|(\what f)^{''}( u)|\le \|h\|_{L^{1}(\R)}
=\int_{\R}|\mathscr A f(s)|s^2\, ds \lesssim \int_{\R}|\mathscr A f(s)|e^{\frac{\gamma_q\rho }{2}s }\, ds \lesssim \|f\|_q.
\end{equation}
Hence from \eqref{flhm200}, \eqref{soe12} and \eqref{seq1}  we get
\begin{equation}\label{ssmeq32}
\|h_t\|_2^{-2}\|f \ast h_t- \what f(0) h_t\|_2^2
\lesssim\frac{\|f\|_q}{t^2}\int_{0}^\infty e^{-2 \lambda^2}
\lambda^{6}\left(1+\frac{\lambda}{\sqrt t}\right)^{{n-3}}  \, d\lambda.
\end{equation}
Multplying by $\theta(t)$ on both sides of  the inequality above, we obtain
\begin{equation}\label{sseq32}
\frac{\theta(t)^2\|f \ast h_t- \what f(0) h_t\|_2^2}{\|h_t\|_2^2}
\lesssim\left(\frac{\theta(t)}{t}\right)^2\int_{0}^\infty e^{-2 \lambda^2}
\lambda^{6}\left(1+\frac{\lambda}{\sqrt t}\right)^{{n-3}}d\lambda.
\end{equation}
Since $\theta(t)/t$ converges to zero as $t$ tends to infinity and for all $t>1$
\begin{equation*}
\left(1+\frac{\lambda}{\sqrt t}\right)^{{n-3}} \le \left(1+{\lambda}\right)^{(n-3)_+},
\end{equation*}
taking limit as $t$ tends to infinity in \eqref{sseq32} we get the desired result.
We now assume that $p >2$ and define $\psi(t)= \theta(2t)$, $t\in (0,\infty)$. From the inequality \eqref{ra1} we have
\begin{equation*}
\frac{\theta(t)\|f \ast h_t-\what f(0)h_t\|_p}{\|h_t\|_p}\lesssim {\psi\left(\frac t2\right)}\frac{\|f\ast h_{t/2}-\what f(0)h_{t/2}\|_2}{\|h_{t/2}\|_2}.
\end{equation*}
The result now follows from the case $p=2$.
\item[(b)] Let $f=h_1$, where $h_1$ is the heat kernel at time $t=1$. From \eqref{hsre1} we get
\begin{eqnarray}
&\,&\|h_t\|_2^{-2}\|f \ast h_t- \what f(0) h_t\|_2^2 \no \\
&\asymp& t^{\frac 3 2} \int_{0}^\infty |e^{-(\lambda ^2+\rho ^2)}-e^{-\rho ^2}|^2 e^{-2t \lambda ^2}
\lambda^{2}(1+ \lambda )^{{n-3}}  \, d\lambda \no \\
&\asymp& t^{\frac 3 2} \int_{0}^\infty (1-e^{-\lambda ^2})^2 e^{-2t \lambda ^2}
\lambda^{2}(1+ \lambda )^{{n-3}}  \, d\lambda \label{hsre1m}.
\end{eqnarray}
Using the elementary estimate
\begin{equation*}
1-e^{-x} \asymp \frac{x}{1+x},\:\:\:\:\:\:x\in [0,\infty),
\end{equation*}
  in \eqref{hsre1m} we get
\begin{eqnarray}
&\,&\|h_t\|_2^{-2}\|f \ast h_t- \what f(0) h_t\|_2^2 \no\\
&\asymp& t^{\frac 3 2} \int_{0}^\infty \frac{\lambda^4}{(1+\lambda^2)^2}e^{-2t \lambda ^2}
\lambda^{2}(1+ \lambda )^{{n-3}}  \, d\lambda  \no
\\ &\asymp& \frac{1}{t^2}\int_{0}^\infty \frac{\lambda^4}{\left(1+\frac{\lambda^2}{t}\right)^2}e^{-2 \lambda ^2}
\lambda^{2}\left(1+ \frac{ \lambda }{\sqrt t}\right)^{{n-3}}  \, d\lambda \label{hsre3m}.
\end{eqnarray}
Hence,  it follows from \eqref{hsre3m} that
\begin{eqnarray}
&\,&\limsup_{t\to \infty}\theta(t)^2\|h_t\|_2^{-2}\|f \ast h_t- \what f(0) h_t\|_2^2 \asymp \limsup_{t \to \infty}\left(\frac{\theta(t)}{t}\right)^2\int_{0}^\infty {\lambda^6}e^{-2 \lambda ^2}
\, d\lambda >0.\no
\end{eqnarray}
\end{itemize}
This completes the proof.\end{proof}

\subsection{Counter example for ball average }
\label{ballc}
We recall that
\begin{equation*}
m_r(x)= \frac {\chi_{B(o, r)}(x)}{|B(o, r)|},\:\:\:\:\:x\in X,
\end{equation*}
and for any $f \in L^1(\R^n)$ (see \eqref{hn})
\begin{equation*}\lim_{r \to \infty}\|f \ast m_r-M(f)m_r\|_{L^1(\R^n)}=0.
\end{equation*}
However, as we will see, this result is not true in rank one symmetric space of noncompact type $X$. Indeed, we will show that  for any $z \in \C$ and for any left $K$-invariant, non-zero, non-negative $f \in L^1(X)$
\begin{equation}
\label{bcq1}
\displaystyle\lim_{r \to \infty}\|m_r\|_p^{-1}\|f \ast m_r- z\,m_r\|_{p}\neq 0,
\end{equation}
for all $p\in[1, \infty]$. Thus, the exact analogue  of Theorem \ref{cl16} is not true on $X$ which vindicates that the heat kernel $h_t$ is a much better kernel than $m_r$ on symmetric spaces in contrasts to the Euclidean spaces where both kernels have similar asymptotic behaviour. In order to show \eqref{bcq1}  we will need the following facts about the kernel $m_r$.

Using  description of the Haar measure \eqref{ch1}-\eqref{ch2} it can be shown that
\begin{equation}
\label{mre34}
\|m_r\|_p \asymp e^{-\frac{ 2\rho r}{p'}},\quad\text{for all $r>1$}.
\end{equation}
Let
\begin{equation*}
\psi_{\lambda}(r)= \int_{X} m_r(x) \varphi_{-\lambda}(x)\, dx
\end{equation*}
be the spherical Fourier transform of $m_r$ at $\lambda$. Since $\varphi_\lambda= \varphi_{-\lambda}$, it is clear that $\psi_\lambda= \psi_{-\lambda}$. It is known that $\psi_\lambda(r)$ can be expressed in terms of the Jacobi functions  \cite[(2.4.3)]{MS-1} and using the asymptotics for Jacobi functions (see \cite[(2.3.2)--(2.3.3)]{MS-2}), it can be shown that for $\lambda \in \C$ with $\Im \lambda < 0$,
\begin{equation}
 \label{baeq1}
  \lim_{r\to \infty }e^{-(i \lambda -\rho)r}\psi_\lambda(r) =  \hc_b(\lambda),
\end{equation}
 where  $\hc_b(\lambda)$  is an analogue of the Harish-Chandra $\hc$-function which has neither zero nor pole in the region  $\Im \lambda<0$.

Now we are ready to prove \eqref{bcq1}. Let $ f \in L^1(G//K)$ be a positive function. As $f$ is positive, it follows that
\begin{equation}
\label{cbs3}
|\what f(a-i\gamma_p\rho)| < \what f(-i\gamma_p\rho),
\end{equation}
for any $a \in \R\setminus\{0\}$, $p\in[1,2)$ (see \cite[Proposition 2.5.2(a)]{MS-1}). We shall prove \eqref{bcq1} case by case.

\noindent {\bf Case I:} Let $p=1$, so that $\gamma_1=1$. Suppose \begin{equation}
\label{cbs1}
\lim_{r \to \infty}\|f \ast m_r- z\,m_r\|_{1}= 0.
\end{equation}
We know that
\begin{equation}
\label{cbs2}
|\what f(a-i\rho)-z||\psi_{a-i\rho}(r)| \le \|f \ast m_r- z\,m_r\|_{1},
\end{equation}
for all $a\in\R$. In view of \eqref{baeq1},  we have
\begin{equation*}
	 \lim_{r \to \infty}|\psi_{a-i\rho}(r)|>0,
\end{equation*}
for all $a\in\R.$ Using this observation together with \eqref{cbs1} and \eqref{cbs2} we get
\begin{equation*}
\what f(a-i\rho)= z,
\end{equation*}
for all $a \in \R$. This clearly contradicts \eqref{cbs3}. Hence,
\begin{equation}
\lim_{r \to \infty}\|f \ast m_r- z\,m_r
\|_{1}\neq 0.
\end{equation}
{\bf Case II:}\:\:
Let $p \in (1,2)$. We fix a sequence $\{r_n\}$ of positive numbers such that $r_n \uparrow \infty$.
Using \eqref{mre34} and \eqref{baeq1}, we note that
\begin{equation}\label{defnliminf}
h(\la)=\liminf_{n \to \infty}\frac{|\psi_{\lambda-i\gamma_p\rho}(r_n)|}{\|m_{r_n}\|_p}>0,\:\:\:\:\:\text{for all}\:\:\la>0.
\end{equation}
Using the Hausdorff--Young inequality \eqref{hyineq} with $F=(f-z)\ast m_r$, we get
\begin{equation}\label{hyineqappl}
	\left(\int_0^\infty |\what f(\lambda-i\gamma_p \rho )-z|^{p'} |\psi_{\lambda-i\gamma_p\rho}(r_n)|^{p'}\tilde{\mu}(\la)\:d\la\right)^{\frac{1}{p'}}\lesssim\|f \ast m_{r_n}- z\,m_{r_n}	\|_p.
\end{equation}
Applying Fatou's Lemma and using the description \eqref{defnliminf} of $h$, we obtain
\begin{eqnarray}\label{fatou}
	&&\int_0^\infty |\what f(\lambda-i\gamma_p \rho )-z|^{p'} h(\lambda)^{p'}\tilde{\mu}(\la)\:d\la\no\\
	&\leq&\liminf_{n \to \infty}\|m_{r_n}\|_p^{-p'}\int_0^\infty |\what f(\lambda-i\gamma_p \rho )-z|^{p'}|\psi_{\lambda-i\gamma_p\rho}(r_n)|^{p'}\tilde{\mu}(\la)\:d\la\no\\
	&\lesssim&\liminf_{n \to \infty}\|m_{r_n}\|_p^{-p'}\|f \ast m_{r_n}- z\,m_{r_n}	\|_p^{p'},
\end{eqnarray}
where we have used \eqref{hyineqappl} in the last step. If
\begin{equation*}
\lim_{r \to \infty}\|m_r\|_p^{-1}\|f \ast m_r- z\,m_r\|_{p}=0,
\end{equation*}
then it follows from \eqref{fatou} and \eqref{defnliminf} that $\what f(\lambda-i\gamma_p \rho )$ equals $z$, for almost every \\$\lambda\in (0,\infty)$. Using continuity of $\what f$ we conclude that for all $\lambda \in (0,\infty)$
\begin{equation*}
\what f(\lambda-i\gamma_p \rho )=z,
\end{equation*}
which again contradicts \eqref{cbs3}.

\noindent {\bf Case III:}\:\:Let $p\in[2,\infty]$. We fix $ q\in[1,2)$, and choose $\theta\in[0,1)$, such that
\begin{equation*}
 q^{-1} = (1- \theta) +\theta p^{-1};
\end{equation*}
that is, $$\frac{1}{q'}=\frac{\theta}{p'}.$$
Using convexity of norms, we get
\begin{eqnarray}
\no
\|f \ast m_r- z\,m_r\|_{q} &\le& \|f \ast m_r- z\,m_r\|^{1-\theta}_{1} \|f \ast m_r- z\,m_r\|^{\theta}_{p}\\& \lesssim &\|m_r\|_{1}^{1-\theta}\|f \ast m_r- z\,m_r\|^{\theta}_{p}\no\\&\lesssim&\|f \ast m_r- z\,m_r\|^{\theta}_{p}. \label{mref}
\end{eqnarray}
From \eqref{mre34} we observe that for all large $r$
\begin{equation*}
  \|m_r\|_q\asymp\|m_r\|_p^{\theta}.
\end{equation*}
Using this observation in \eqref{mref} we obtain
\begin{eqnarray}
\label{mref1}
\|m_r\|_q^{-1}\|f \ast m_r- z\,m_r\|^q_{q} \lesssim \|m_r\|_p^{-\theta}\|f \ast m_r- z\,m_r\|^{\theta}_{p}.
\end{eqnarray}
Now, if  $$\lim_{r \to \infty}\|m_r\|_p^{-1}\|f \ast m_r- z\,m_r\|_{p}=0,$$
from  \eqref{mref1} it follows that $$\lim_{r \to \infty}\|m_r\|_q^{-1}\|f \ast m_r- z\,m_r\|_{q}=0,$$ which is clearly not true in view of Case I and Case II, as $q\in[1,2)$. Thus,
\begin{equation}
	\displaystyle\lim_{r \to \infty}\|m_r\|_p^{-1}\|f \ast m_r- z\,m_r\|_{p}\neq 0.
\end{equation}

\section*{Acknowledgments}
The authors thank the referee for a careful reading of the manuscript and for the valuable comments, which have improved the exposition of the paper. The authors also wish to thank Mithun Bhowmik, Manjunath Krishnapur, and Sundaram Thangavelu for several useful discussions. Muna Naik and Jayanta Sarkar are supported by the INSPIRE faculty fellowship (Faculty Registration No.: IFA20-MA151, Faculty Registration No.: IFA22-MA-172) from the Department of Science and Technology, Government of India.


\begin{thebibliography}{99}

\bibitem{A} Anker, JP. {\em La forme exacte de l'estimation fondamentale de Harish-Chandra.} C. R. Acad. Sci. Paris Sér. I Math. 305 (1987), no. 9, 371–374

\bibitem{AL}Anker, JP. , Ji, L. {\em Heat kernel and Green function estimates on noncompact symmetric spaces. } Geom. Funct. Anal. \textbf{9}, 1035-1091 (1999), \url{https://doi.org/10.1007/s000390050107}

\bibitem{AO}  Anker, JP.; Ostellari, P. {\em The heat kernel on noncompact symmetric spaces}. Lie groups and symmetric spaces, 27–46, Amer. Math. Soc. Transl. Ser. 2, 210, Adv. Math. Sci., 54, Amer. Math. Soc., Providence, RI, 2003.

\bibitem{AE}Anker, JP., Papageorgiou, E. \& Zhang, H. {\em Asymptotic behavior of solutions to the heat equation on noncompact symmetric spaces.} J. Funct. Anal., 284, Paper No. 109828, 43 (2023), \url{https://doi.org/10.1016/j.jfa.2022.109828}

\bibitem{Cow-h} Cowling, M. {\em Herz’s “principe de majoration” and the Kunze--Stein phenomenon}. Harmonic analysis and Number Theory, CMS Conf. Proc. 21, A. M. S., Providence, RI, 1997, pp. 73--88.

\bibitem{Cow1} Cowling, M., Giulini, S., Meda, S.  {\em $L^p-L^q$-Estimates for functions of the Laplace--Beltrami operator on noncompact symmetric spaces, I*.} Duke Math. J. 72 (1993), no. 1, 109–150.

\bibitem{Cow2}Cowling, M., Giulini, S., Meda, S.  {\em $L^p-L^q$-Estimates for functions of the Laplace--Beltrami operator on noncompact symmetric spaces, II*.} Journal of Lie Theory. Volume 5 (1995), 1-14.

\bibitem{GV} Gangolli, R., Varadarajan, V. S. {\em Harmonic analysis of spherical functions on real reductive groups.} Ergebnisse der Mathematik und ihrer Grenzgebiete, 101. Springer-Verlag, Berlin, 1988.

\bibitem{Grac} Graczyk, P., St\'{o}s, A. {\em Transition Density Estimates for stable processes on symmetric spaces.} Pacific Journal of Mathematics, Vol. 217, No. 1, 2004.

\bibitem{Net}Hansen, W., Netuka, I. {\em Successive averages and harmonic functions.} J. Anal. Math. \textbf{71} pp. 159-171 (1997), \url{https://doi.org/10.1007/BF02788028}


\bibitem{Helga-3}Helgason, S. {\em Groups and geometric analysis. Integral geometry, invariant differential operators, and spherical functions.} Corrected reprint of the 1984 original. Mathematical Surveys and Monographs, 83. American Mathematical Society, Providence, RI, 2000.

\bibitem{Helga-2}Helgason, S. {\em Geometric analysis on symmetric spaces. Second edition. } Mathematical Surveys and Monographs, 83. American Mathematical Society, Providence, RI, 2008.

\bibitem{Her} Herz, C. {\em Problems of extrapolation and spectral synthesis on groups. {\em Conference on Harmonic Analysis.} Univ. Maryland, College Park, Md. 1971, pp. 157–166, Lecture Notes in Math., Vol. 266, Springer, Berlin-New York, 1972}.

\bibitem{MS-2}Naik, M.,  Sarkar, R.P. {\em Asymptotic mean value property for eigenfunctions of the Laplace-Beltrami operator on  Damek-Ricci spaces.}  Ann. Mat. Pura Appl. (4) 201 (2022), 1583--1605.

\bibitem{MS-1}Naik, M., Sarkar, R.P.  {\em Characterization of eigenfunctions of the Laplace-Beltrami operator using Fourier multipliers.} J. Funct. Anal. \textbf{279}, 108737, 43 (2020), \url{https://doi.org/10.1016/j.jfa.2020.108737}



\bibitem{Eff3}Papageorgiou, E. {\em $L^p$ asymptotics for the heat equation on symmetric spaces for non-symmetric solutions}, Int. Math. Res. Not. IMRN 2025, no. 7, Paper No. rnaf074, 26 pp.
\url{https://doi.org/10.1093/imrn/rnaf074}

\bibitem{Eff2}Papageorgiou, E.  {\em Asymptotic behavior of solutions to the extension problem for the fractional Laplacian on noncompact symmetric spaces.} J. Evol. Equ. 24, 34 (2024). \url{https://doi.org/10.1007/s00028-024-00959-6}

\bibitem{Eff1}Papageorgiou, E.  {\em Large-Time Behavior of Two Families of Operators Related to the Fractional Laplacian on Certain Riemannian Manifolds.} Potential Anal. (2023), \url{https://doi.org/10.1007/s11118-023-10109-1}


\bibitem{ST} Stanton, R.J., Tomas, P.A. {\em Pointwise inversion of the spherical transform on $L^p(G/K)$, $1\leq p<2$.} Proc. Amer. Math. Soc. 73 (1979), no. 3, 398–404.

\bibitem{Vaz-3}Vázquez, J.L.  {\em Asymptotic behaviour methods for the heat equation. Convergence to the Gaussian}, preprint, arXiv :1706.10034, 2018.

\bibitem{Vaz-1}Vázquez, J.L. {\em Asymptotic behaviour for the fractional heat equation in the Euclidean space.}  Complex Var. Elliptic Equ. \textbf{63}, 1216-1231 (2018).

\bibitem{Vaz-2}Vázquez, J.L. {\em  Asymptotic behaviour for the heat equation in hyperbolic space}. Comm. Anal. Geom. \textbf{30}, 2123-2156 (2022).


\end{thebibliography}
\end{document}